\documentclass{amsart}
\usepackage[all,ps,cmtip]{xy}
\usepackage{amsmath, amssymb}
\usepackage{amscd}
\usepackage{tikz,color}

\newtheorem{theorem}{Theorem}[section]
\newtheorem{lemma}[theorem]{Lemma}

\newtheorem{corollary}[theorem]{Corollary}

\newtheorem{proposition}[theorem]{Proposition}
\newtheorem{conjecture}[theorem]{Conjecture}

\theoremstyle{definition}
\newtheorem{definition}[theorem]{Definition}

\newtheorem{question}[theorem]{Question}

\theoremstyle{remark}
\newtheorem{remark}[theorem]{Remark}

\numberwithin{equation}{section}

\DeclareMathOperator{\rank}{rk}

\DeclareMathOperator{\codim}{codim}

\DeclareMathOperator{\Hom}{Hom}

\DeclareMathOperator{\Coh}{Coh}

\DeclareMathOperator{\D}{D}

\DeclareMathOperator{\K}{K}

\DeclareMathOperator{\Supp}{Supp}

\DeclareMathOperator{\ch}{ch}

\DeclareMathOperator{\Ext}{Ext}

\begin{document}
\title{Tilt-stability, vanishing theorems and Bogomolov-Gieseker type inequalities}

\author{Hao Max Sun}

\address{Department of Mathematics, Shanghai Normal University, Shanghai 200234, People's Republic of China}

\email{hsun@shnu.edu.cn}



\date{August 2, 2016}


\begin{abstract}
We investigate the tilt-stability of stable sheaves on projective
varieties with respect to certain tilt-stability conditions depends
on two parameters constructed by Bridgeland \cite{Bri2} (see also
\cite{AB, BMT, BMS}). For a stable sheaf, we give effective bounds
of these parameters such that the stable sheaf is tilt-stable. These
allow us to prove new vanishing theorems for stable sheaves and an
effective Serre vanishing theorem for torsion free sheaves. Using
these results, we also prove Bogomolov-Gieseker type inequalities
for the third Chern character of a stable sheaf on $\mathbb{P}^3$.
\end{abstract}

\maketitle
\section{Introduction}
Let $X$ be a complex smooth projective variety of dimension $n$ with
a fixed ample divisor $H$ and a fixed $\mathbb{Q}$-divisor $B$ on
it. For any real numbers $\alpha>0$ and $\beta$, the
$\mathbb{R}$-divisors $\alpha H$ and $\beta H+B$ determine a weak
Bridgeland stability condition on $X$ (see Section \ref{S2.2} for
the precise definition). We also call it
$\nu_{\alpha,\beta}$-stability (or tilt-stability). In recent years,
this stability has drawn a lot of attentions, and has been
investigated intensively.

When $X$ is a surface, $\nu_{\alpha,\beta}$-stability is a
Bridgeland stability condition introduced by Bridgeland \cite{Bri1,
Bri2}, Arcara and Bertram \cite{AB}. There are many and fruitful
applications of this stability to birational geometry of moduli
spaces of stable sheaves on surfaces (cf. \cite{AB}, \cite{ABCH},
\cite{BM1}, \cite{BM2}, \cite{Bo}, \cite{MYY}, \cite{Nu} \cite{LZ},
$\cdots$). For higher dimensional $X$,
$\nu_{\alpha,\beta}$-stability appears in the construction of
Bridgeland stability on $X$ by Bayer, Macr\`i and Toda \cite{BMT},
and it has been systematically investigated by Bayer, Macr\`i and
Stellari \cite{BMS}.

The prototypical example of a $\nu_{\alpha,\beta}$-stability result
is Bridgeland's large volume limit theorem \cite{Bri2, Liu}:
\begin{theorem}[Bridgeland]
Suppose that $\dim X=2$. For $E\in \Coh^{\beta H+B}(X)\cap\Coh(X)$
and $\alpha\gg0$, we have $E$ is $\nu_{\alpha, \beta}$-(semi)stable
if and only if $E$ is $(H,\beta H+B-\frac{1}{2}K_X)$-twisted
Gieseker (semi)stable.
\end{theorem}
A parallel result also holds for higher dimensional case. In that
case, the large volume limit $(\alpha\gg0)$ of
$\nu_{\alpha,\beta}$-stability for a coherent sheaf is the same as
the $p_{H, \beta H+B}$-stability (see Proposition \ref{prop1}).

Bridgeland's arguments are non-effective, and there is no known
bound on how large $\alpha$ must be in order to obtain the
conclusion of the theorem. In light of this theorem, it is natural
to ask:
\begin{question}
For which finite value of $\alpha$ and $\beta$ does $p_{H, \beta
H+B}$-stability become $\nu_{\alpha,\beta}$-stability for a coherent
sheaf?
\end{question}

The goal of this paper is to answer this question for a
$\mu_{H,B}$-stable torsion free sheaf $E$ (see Definition
\ref{def2.1}). In order to investigate the
$\nu_{\alpha,\beta}$-stability of $E$, all the potential
$\nu_{\alpha,\beta}$-destabilizing subsheaves of $E$ should be
considered. We need to define $$\mu^{\max}_{H,
B}(E)=\max\Big\{\mu_{H,B}(F): F~\mbox{is a subsheaf of}~E,
\mu_{H,B}(F)\neq\mu_{H,B}(E)\Big\}.$$ If $F$ is a subsheaf of $E$
with $\mu_{H,B}(F)=\mu_{H,B}(E)$, since $E$ is $\mu_{H,B}$-stable,
one sees that $\rank F=\rank E$. This implies that $E/F$ is a
torsion sheaf, and the codimension of the support of $E/F$ is $\geq$
2. It follows that $\nu_{\alpha,\beta}(F)\leq\nu_{\alpha,\beta}(E)$,
namely that such an $F$ can not $\nu_{\alpha,\beta}$-destabilize
$E$. Our main result is the following.

\begin{theorem}[=Theorem \ref{thm1}]\label{main1}
Suppose that $E$ is a $\mu_{H, B}$-stable torsion free sheaf on $X$,
and $\mu$ is a rational number satisfies $\mu^{\max}_{H,
B}(E)\leq\mu<\mu_{H,B}(E)$. Let
$\beta_0=\mu_{H,B}(E)-\frac{\overline{\Delta}_H^B(E)/(H^n\rank
E)^2}{\mu_{H,B}(E)-\mu}$ and
$\beta_1=\mu_{H,B}(E)-\frac{\sqrt{(\rank
E+1)\overline{\Delta}_H^B(E)}}{H^n\rank E}$.
\begin{enumerate}
\item If $\mu>\mu_{H, B}(E)-\frac{1}{H^n\rank
E}\sqrt{\frac{\overline{\Delta}_H^B(E)}{\rank E+1}}$, then $E$ is
$\nu_{\alpha, \beta}$-stable for any $\alpha>0$ and
$\beta\leq\beta_0$.

\item If $\mu\leq\mu_{H, B}(E)-\frac{1}{H^n\rank
E}\sqrt{\frac{\overline{\Delta}_H^B(E)}{\rank E+1}}$ and
$\overline{\Delta}_H^B(E)>0$, then $E$ is $\nu_{\alpha,
\beta_1}$-stable for any $\alpha>0$.

\item If $\overline{\Delta}_H^B(E)=0$, then $E$ is $\nu_{\alpha,
\beta}$-stable for any $\alpha>0$ and $\beta<\mu_{H,B}(E)$.
\end{enumerate}
\end{theorem}
Generally it is difficult to compute $\mu^{\max}_{H, B}(E)$. That
prevents the applications of the theorem. But one can obtain some
upper bounds for $\mu^{\max}_{H, B}(E)$ explicitly. Taking $\mu$ to
be those upper bounds can make the theorem to be conveniently
applied.

Considering the $\nu_{\alpha,\beta}$-stability of $E[1]$, we have
the following result dual to Theorem \ref{main1}.

\begin{theorem}[=Theorem \ref{thm2}]\label{main2}
Suppose that $E$ is a $\mu_{H,B}$-stable reflexive sheaf on $X$, and
$\bar{\mu}$ is a rational number satisfies $\mu_{H,
B}(E)<\bar{\mu}\leq\mu^{\min}_{H,B}(E)$. Let
$\overline{\beta}_0=\mu_{H,B}(E)+\frac{\overline{\Delta}_H^B(E)/(H^n\rank
E)^2}{\bar{\mu}-\mu_{H, B}(E)}$ and
$\overline{\beta}_1=\mu_{H,B}(E)+\frac{\sqrt{(\rank
E+1)\overline{\Delta}_H^B(E)}}{H^n\rank E}$.
\begin{enumerate}
\item If $\bar{\mu}<\mu_{H, B}(E)+\frac{1}{H^n\rank
E}\sqrt{\frac{\overline{\Delta}_H^B(E)}{\rank E+1}}$, then $E[1]$ is
$\nu_{\alpha, \beta}$-stable for any $\alpha>0$ and
$\beta\geq\overline{\beta}_0$.

\item If $\bar{\mu}\geq\mu_{H, B}(E)+\frac{1}{H^n\rank
E}\sqrt{\frac{\overline{\Delta}_H^B(E)}{\rank E+1}}$, then $E[1]$ is
$\nu_{\alpha, \overline{\beta}_1}$-stable for any $\alpha>0$.

\item If $\overline{\Delta}_H^B(E)=0$, then $E[1]$ is
$\nu_{\alpha,\beta}$-stable for any $\alpha>0$ and $\beta\geq\mu_{H,
B}(E)$.
\end{enumerate}
\end{theorem}

Theorem \ref{main1}(3) and \ref{main2}(3) are well known (see
\cite[Lemma 6.28]{MS}), but we still state them in order to keep
this note self-complete. The above theorems can help us to
understand the $\nu_{\alpha,\beta}$-stability more explicitly. They
also give some interesting applications to the positivity of
coherent sheaves, such as vanishing theorems of stable sheaves,
effective Serre vanishing theorem, and Bogomolov-Gieseker type
inequalities for the third Chern character of a stable sheaf on
$\mathbb{P}^3$.

The slope $\mu^{\min}_{H,B}(E)$ in Theorem \ref{main2} is defined in
Section \ref{S5}. The strategy of the proof is essentially the same
as that of \cite{Sun}.

We now explain the strategy in greater detail. Given a
$\mu_{H,B}$-stable sheaf $E$, we define an ellipse $C_E$ on the
$(\beta,\alpha)$ half plane. By the Bogomolov-Gieseker type
inequality in \cite{BMT} and \cite{BMS}, we can show that for a
point $(\beta,\alpha)$ outside the ellipse $C_E$, if a subobject $F$
of $E$ has large $\nu_{\alpha,\beta}$-slope, then it must have small
rank (Lemma \ref{lemma4.1} and Lemma \ref{lemma4.2}). When $F$ has
small rank, by the $\mu_{H,B}$-stability of $E$ and the
Bogomolov-Gieseker type inequality, $\nu_{\alpha,\beta}(F)$ can be
bounded above. Hence we could obtain the
$\nu_{\alpha,\beta}$-stability of $E$ by computing the intersection
of the ellipse and the wall $W(F,E)$.

\subsection*{Applications to vanishing theorems}
By the basic properties of $\nu_{\alpha,\beta}$-stability, Theorem
\ref{main1} and \ref{main2} can immediately give the following
vanishing theorems for $\mu_{H, B}$-stable sheaves.

\begin{corollary}\label{main3}
Let $E$ be a $\mu_{H}$-stable torsion free sheaf on $X$, and $\mu$
be a rational number satisfies
$\mu^{\max}_{H}(E)\leq\mu<\mu_{H}(E)$.
\begin{enumerate}
\item If $\mu>\mu_{H}(E)-\frac{1}{H^n\rank
E}\sqrt{\frac{\overline{\Delta}_H(E)}{\rank E+1}}$, then $H^{n-1}(X,
E(K_X+lH))=0$ for any integer
$l>\frac{\overline{\Delta}_H(E)/(H^n\rank
E)^2}{\mu_{H}(E)-\mu}-\mu_{H}(E)$.

\item If $\mu\leq\mu_{H}(E)-\frac{1}{H^n\rank
E}\sqrt{\frac{\overline{\Delta}_H(E)}{\rank E+1}}$, then $H^{n-1}(X,
E(K_X+lH))=0$ for any integer $l>\frac{\sqrt{(\rank
E+1)\overline{\Delta}_H(E)}}{H^n\rank E}-\mu_{H}(E)$.
\end{enumerate}
\end{corollary}

\begin{corollary}\label{main4}
Let $E$ be a $\mu_{H}$-stable reflexive sheaf on $X$, and
$\bar{\mu}$ be a rational number satisfies
$\mu_{H}(E)<\bar{\mu}\leq\mu^{\min}_{H}(E)$.
\begin{enumerate}
\item If $\bar{\mu}<\mu_{H}(E)+\frac{1}{H^n\rank
E}\sqrt{\frac{\overline{\Delta}_H(E)}{\rank E+1}}$, then $H^1(X,
E(-lH))=0$ for any integer
$l>\mu_{H}(E)+\frac{\overline{\Delta}_H(E)/(H^n\rank
E)^2}{\bar{\mu}-\mu_{H}(E)}$.

\item If $\bar{\mu}\geq\mu_{H}(E)+\frac{1}{H^n\rank
E}\sqrt{\frac{\overline{\Delta}_H(E)}{\rank E+1}}$, then $H^1(X,
E(-lH))=0$ for any integer $l>\mu_{H}(E)+\frac{\sqrt{(\rank
E+1)\overline{\Delta}_H(E)}}{H^n\rank E}$.
\end{enumerate}
\end{corollary}

These vanishing theorems generalize the Kodaira vanishing. To see
this, just taking $E=\mathcal{O}_X(H)$ in Corollary \ref{main3} and
Corollary \ref{main4}, then one obtains the Kodaira vanishing
$$H^{n-1}(X, \mathcal{O}_X(K_X+H))=H^1(X, \mathcal{O}_X(-H))=0.$$

These vanishing theorems can be used to give an effective Serre
vanishing theorem for $H^{n-1}$.

\begin{theorem}[Effective Serre vanishing for $H^{n-1}$]\label{Serre}
Let $\mathcal{F}$ be a coherent torsion free sheaf on $X$, and let
$0=\mathcal{F}_0\subset \mathcal{F}_1\subset\cdots\subset
\mathcal{F}_k=\mathcal{F}$ be its Harder-Narasimhan filtration. Set
$\mathcal{G}_i=\mathcal{F}_i/\mathcal{F}_{i-1}$. Then $H^{n-1}(X,
\mathcal{F}(K_X+lH))=0$ as soon as
\begin{eqnarray*}
l>\max_{1\leq i\leq
k}\Big\{\frac{\overline{\Delta}_H(\mathcal{G}_i)}{H^n}-\mu_{H}(\mathcal{G}_i),
\sqrt{\frac{2\overline{\Delta}_H(\mathcal{G}_i)}{(H^n)^2\rank
\mathcal{G}_i}}-\mu_{H}(\mathcal{G}_i)\Big\}.
\end{eqnarray*}
\end{theorem}

To the best of our knowledge, no explicit bounds for such an $l$ in
Theorem \ref{Serre} are known before except for some special cases.
In the case that $H$ is very ample, Langer \cite[Theorem
0.1]{Langer} also gives an effective bounds for Serre vanishing
theorem in the surface case (his proof works in any characteristic).
It is clear that for very ample $H$, the vanishing of $H^{n-1}$ in
higher dimensions also follows from the surface case by a very
simple restriction arguments. See \cite{Reider}, \cite{BS} and
\cite{Tan} for such an effective bound for rank one torsion free
sheaves on a surface. See also \cite[Example 10.2.9]{Laz} for the
effective bound for an ample line bundle on a projective variety of
any dimension.

In particular, for a $\mu_{H}$-semistable torsion free sheaf one
has:
\begin{corollary}\label{corv}
Let $\mathcal{F}$ be a $\mu_{H}$-semistable torsion free sheaf on
$X$ with $\rank \mathcal{F}\geq2$. Then $H^{n-1}(X,
\mathcal{F}(K_X+lH))=0$ if
$l>\frac{\overline{\Delta}_H(\mathcal{F})}{H^n}-\mu_{H}(\mathcal{F})$.
\end{corollary}

In the case that $H$ is very ample and $\dim X=2$, Langer
\cite[Theorem 0.1]{Langer} gives the vanishing of $H^{1}(X,
\mathcal{F}(lH))=0$ for $$l>\frac{\Delta(\mathcal{F})}{2\rank
\mathcal{F}}-\mu_{H}(\mathcal{F})+\mbox{a linear function of}~\rank
\mathcal{F}.$$ One notes that the bound in Corollary \ref{corv} is
of the same form but with somewhat different coefficients at
$\Delta(\mathcal{F})$ and $\rank \mathcal{F}$.

\subsection*{Applications to stable sheaves on $\mathbb{P}^3$}
From the Bogomolov-Gieseker type inequality on $\mathbb{P}^3$ proved
by Macr\`i (Theorem \ref{*}), one could obtain another application
of Theorem \ref{main1} to stable sheaves on $\mathbb{P}^3$. In order
to state them explicitly, we need the following function.

\begin{definition}\label{round}
Let $r$ be a real number, and $m$ be a positive integer, we define
$$[r]_{m}:=\max\{\frac{a}{b}: a, b\in \mathbb{Z}, \frac{a}{b}<r,
1\leq b\leq m\}.$$
\end{definition}

\begin{theorem}\label{Chern}
Suppose $H$ is a plane in $\mathbb{P}^3$. Let $E$ be a
$\mu_{H}$-stable torsion free sheaf on $\mathbb{P}^3$, and let
$l(E)=\frac{c_1^3(E)-3c_1(E)\Delta(E)}{6(\rank E)^2}$.
\begin{enumerate}
\item If $\mu^{\max}_{H}(E)>\mu_{H}(E)-\frac{1}{\rank
E}\sqrt{\frac{\Delta(E)}{\rank E+1}}$, then
\begin{equation*} \frac{\Delta(E)}{6\rank
E}\Big(\mu_H(E)-[\mu_H(E)]_{\rank E}+\frac{\Delta(E)/(\rank
E)^2}{\mu_H(E)-[\mu_H(E)]_{\rank E}}\Big)+l(E)\geq\ch_3(E).
\end{equation*}
\item If $\mu^{\max}_{H}(E)\leq\mu_{H}(E)-\frac{1}{\rank
E}\sqrt{\frac{\Delta(E)}{\rank E+1}}$, then
\begin{equation*}
\frac{(\rank E+2)(\Delta(E))^{\frac{3}{2}}}{6(\rank E)^2\sqrt{\rank
E+1}}+l(E)\geq\ch_3(E).
\end{equation*}
\end{enumerate}
\end{theorem}

In particular, when $\rank E=2$, we have:
\begin{corollary}\label{cor}
Under the situation in the above theorem, we further assume that $E$
is of rank two.
\begin{enumerate}
\item If $\mu^{\max}_{H}(E)>-\sqrt{\frac{c_2(E)}{3}}$ and $c_1(E)=0$, then
$c_3(E)\leq\frac{4}{3}c_2^2(E)+\frac{1}{3}c_2(E)$.

\item If $\mu^{\max}_{H}(E)\leq-\sqrt{\frac{c_2(E)}{3}}$ and $c_1(E)=0$, then
$c_3(E)\leq\big(\frac{4}{3}c_2(E)\big)^{\frac{3}{2}}$.

\item If $\mu^{\max}_{H}(E)>-\frac{1}{2}-\frac{1}{2}\sqrt{\frac{4c_2(E)-1}{3}}$ and $c_1(E)=-1$, then
$c_3(E)\leq\frac{4}{3}c_2^2(E)-\frac{1}{3}c_2(E)$.

\item If $\mu^{\max}_{H}(E)\leq-\frac{1}{2}-\frac{1}{2}\sqrt{\frac{4c_2(E)-1}{3}}$ and $c_1(E)=-1$, then
$c_3(E)\leq\big(\frac{4c_2(E)-1}{3}\big)^{\frac{3}{2}}$.
\end{enumerate}
\end{corollary}
One can compare Corollary \ref{cor} with Hartshorne's bounds:
\begin{theorem}[Hartshorne {\cite[Theorem 8.2]{Har}}]
Let $E$ be a rank two $\mu_H$-stable reflexive sheaf on
$\mathbb{P}^3$, where $H$ is a plane in $\mathbb{P}^3$.
\begin{enumerate}
\item If $c_1(E)=0$, then $c_3(E)\leq c_2^2(E)-c_2(E)+2$.

\item If $c_1(E)=-1$, then $c_3(E)\leq c_2^2(E)$.
\end{enumerate}
\end{theorem}
We notice that when
$\mu^{\max}_{H}(E)>\mu_{H}(E)-\frac{1}{2}\sqrt{\frac{\overline{\Delta}_H(E)}{3}}$,
Corollary \ref{cor} is weaker than Hartshorne's result, but it works
for the more general torsion free case. There are also analogous
results in the rank 3 case (see \cite[Theorem 4.2 and 4.3]{EHV} and
\cite[Theorem 1.2]{Miro}). In the rank 2 and 3 cases, the results
\cite{Har, EHV, Miro} are more precise and they describe all the
possible Chern classes of stable sheaves, constructing them for
given triples of Chern classes in the allowed regions.

The bounds in Theorem \ref{Chern} can have a simpler but weaker
form:
\begin{corollary}\label{cor1}
Let $E$ be a $\mu_H$-stable torsion free sheaf on $\mathbb{P}^3$
with $\rank E\geq3$, where $H$ is a plane in $\mathbb{P}^3$. We have
$$\frac{(\Delta(E))^2}{6\rank
E}+\frac{\Delta(E)}{6(\rank
E)^3}+\frac{c_1^3(E)-3c_1(E)\Delta(E)}{6(\rank E)^2}\geq\ch_3(E).$$
\end{corollary}

\subsection*{Organization of the paper}
Our paper is organized as follows. In Section \ref{S2}, we review
basic notions and properties of some classical stabilities for
coherent sheaves, tilt-stability, the conjectural inequality
proposed in \cite{BMT, BMS} and variants of the classical
Bogomolov-Gieseker inequality satisfies by tilt-stable objects. Then
in Section \ref{S3} we recall the properties of walls for
tilt-stability. In Section \ref{S4} we introduce the extremal
ellipses, and study the intersections of the walls and the extremal
ellipses. We prove Theorem \ref{main1} and \ref{main2} in Section
\ref{S5}. Vanishing theorems in Corollary \ref{main3}, Corollary
\ref{main4} and their applications will be showed in Section
\ref{S6}. In Section \ref{S7}, we give the application of Theorem
\ref{main1} to the Chern classes of a $\mu_{H}$-stable sheaf on
$\mathbb{P}^3$.

\subsection*{Notation}
We work over the complex numbers in this paper. We will always
denote by $X$ a smooth projective variety of dimension $n\geq 2$ and
by $\D^b(X)$ its bounded derived category of coherent sheaves. $K_X$
and $\omega_X$ denote the canonical divisor and canonical sheaf of
$X$, respectively. For a triangulated category $\mathcal{D}$, we
write $\K(\mathcal{D})$ for the Grothendieck group of $\mathcal{D}$.

We write $\mathcal{H}^j(E)$ ($j\in \mathbb{Z}$) for the cohomology
sheaves of a complex $E\in \D^b(X)$. We also write $H^j(F)$ ($j\in
\mathbb{Z}_{\geq0}$) for the cohomology groups of a sheaf
$F\in\Coh(X)$. For a sheaf $G\in\Coh(X)$, we denote by $G^*:=\Hom(G,
\mathcal{O}_X)$ the dual sheaf of $G$. Given a complex number
$z\in\mathbb{C}$, we denote its real and imaginary part by $\Re z$
and $\Im z$, respectively.


\subsection*{Acknowledgments}
I would like to thank the anonymous referee for his or her valuable
comments and suggestions which make our paper more readable. I would
also like to thank Rong Du, Zhan Li, Wenfei Liu, Jun Lu, Xin Lu,
Xiaotao Sun, Sheng-Li Tan, Wanyuan Xu, Qizheng Yin, Fei Yu, Xun Yu,
and Lei Zhang for their interest and discussions. The author was
supported by National Natural Science Foundation of China (Grant No.
11771294, 11301201).

\section{Preliminaries}\label{S2}
In this section, we review some basic notions of stability for
coherent sheaves, the weak Bridgeland stability conditions and
Bogomolov-Gieseker type inequalities.

\subsection{Stability for sheaves}
For any $\mathbb{Q}$-divisor $D$ on $X$, we define the twisted Chern
character $\ch^{D}=e^{-D}\ch$. More explicitly, we have
\begin{eqnarray*}
\begin{array}{lcl}
\ch^{D}_0=\ch_0=\rank  && \ch^{D}_2=\ch_2-D\ch_1+\frac{D^2}{2}\ch_0\\
&&\\
\ch^{D}_1=\ch_1-D\ch_0 &&
\ch^{D}_3=\ch_3-D\ch_2+\frac{D}{2}\ch_1-\frac{D^3}{6}\ch_0.
\end{array}
\end{eqnarray*}

The first important notion of stability for a sheaf is slope
stability, also known as Mumford stability. We define the slope
$\mu_{H, D}$ of a coherent sheaf $E\in \Coh(X)$ by
\begin{eqnarray*}
\mu_{H, D}(E)= \left\{
\begin{array}{lcl}
+\infty,  & &\mbox{if}~\ch^D_0(E)=0,\\
&&\\
\frac{H^{n-1}\ch_1^{D}(E)}{H^n\ch_0^{D}(E)}, & &\mbox{otherwise}.
\end{array}\right.
\end{eqnarray*}

\begin{definition}\label{def2.1}
A coherent sheaf $E$ on $X$ is $\mu_{H, D}$-(semi)stable (or
slope-(semi)stable) if, for all non-zero subsheaves
$F\hookrightarrow E$, we have
$$\mu_{H, D}(F)<(\leq)\mu_{H, D}(E/F).$$
\end{definition}
Note that $\mu_{H, D}$ only differs from $\mu_{H}:=\mu_{H, 0}$ by a
constant, thus $\mu_{H, D}$-stability and $\mu_{H}$-stability
coincide. Harder-Narasimhan filtrations (HN-filtrations, for short)
with respect to $\mu_{H, D}$-stability exist in $\Coh(X)$: given a
non-zero sheaf $E\in\Coh(X)$, there is a filtration
$$0=E_0\subset E_1\subset\cdots\subset E_k=E$$
such that: $G_i:=E_i/E_{i-1}$ is $\mu_{H, D}$-semistable, and
$\mu_{H, D}(G_1)>\cdots>\mu_{H, D}(G_k)$. We set $\mu^+_{H,
D}(E):=\mu_{H, D}(G_1)$ and $\mu^-_{H, D}(E):=\mu_{H, D}(G_k)$.

Another well-known stability for a sheaf is Gieseker stability. To
define it, write the reduced twisted Hilbert polynomial of a
positive rank sheaf $E$ as

$$G_{H, D}(E, m)=\frac{\chi(E(mH-D))}{\rank E},$$ where the Euler characteristic is computed
formally. A simple Riemann-Roch computation shows that
\begin{eqnarray*}
G_{H, D}(E,
m)&=&\frac{m^nH^n}{n!}+\frac{m^{n-1}H^{n-1}}{(n-1)!}\Big(\frac{\ch^D_1(E)-\frac{K_X}{2}\rank
E}{\rank E}\Big)\\
&&+\frac{m^{n-2}H^{n-2}}{(n-2)!}\Big(\frac{\ch^D_2(E)-\frac{K_X}{2}\ch_1^D(E)}{\rank
E}+\frac{K_X^2+c_2(X)}{12} \Big)+\cdots
\end{eqnarray*}

\begin{definition}
A coherent torsion free sheaf $E$ on $X$ is $(H,D)$-twisted Gieseker
(semi)stable (or $G_{H, D}$-(semi)stable) if, for all non-zero
proper subsheaves $F\hookrightarrow E$, we have
$$G_{H, D}(F, m)<(\leq)G_{H, D}(E, m)$$ for all $m\gg0$.
\end{definition}
When $D=0$, we recover usual $H$-Gieseker stability.

Now we introduce the $p_{H, D}$-stability mentioned in the
introduction. The polynomial slope $p_{H, D}$ of a sheaf $E\in
\Coh(X)$ is
\begin{eqnarray*}
p_{H, D}(E,m)= \left\{
\begin{array}{lcl}
(+\infty)m+(+\infty),  & &\mbox{if}~\ch^D_0(E)=0,\\
&&\\
\frac{H^{n-1}\ch_1^D(E)}{H^n\ch_0^D
(E)}m+\frac{H^{n-2}\ch_2^D(E)}{H^n\ch_0^D (E)}, & &\mbox{otherwise}.
\end{array}\right.
\end{eqnarray*}

\begin{definition}
A coherent sheaf $E$ on $X$ is $p_{H, D}$-(semi)stable (or
polynomial slope-(semi)stable) if, for all non-zero subsheaves
$F\hookrightarrow E$, we have
$$p_{H, D}(F,m)<(\leq)p_{H, D}(E/F,m)$$ for all $m\gg0$.
\end{definition}

There are obvious relations among those stabilities. One can easily
prove
\begin{lemma}
For any coherent torsion free sheaf $E$ on $X$, one has the
following chain of implications
\begin{eqnarray*}
&&\mu_{H}\mbox{-stable}\Rightarrow
p_{H,D+\frac{K_X}{2}}\mbox{-stable}\Rightarrow G_{H,
D}\mbox{-stable} \Rightarrow G_{H, D}\mbox{-semistable}\\
&&\Rightarrow p_{H,D+\frac{K_X}{2}}\mbox{-semistable}\Rightarrow
\mu_{H}\mbox{-semistable}.
\end{eqnarray*}
Moreover, $G_{H, D}$-stability and $p_{H,D+\frac{K_X}{2}}$-stability
are equivalent for any coherent torsion free sheaf on a surface.
\end{lemma}

\subsection{Weak Bridgeland stability conditions}\label{S2.2}
 The notion of ``weak
Bridgeland stability condition'' and its variant ``very weak
Bridgeland stability condition'' have been introduced in
\cite[Section 2]{Toda1} and \cite[Definition 12.1]{BMS},
respectively. We will use a slightly different notion in order to
adapt our situation. The main difference is the rotation of the
half-plane in $\mathbb{C}$.
\begin{definition}
A weak Bridgeland stability condition on $X$ is a pair $\sigma=(Z,
\mathcal{A})$, where where $\mathcal{A}$ is the heart of a bounded
$t$-structure on $\D^b(X)$, and $Z:\K(\D^b(X))\rightarrow
\mathbb{C}$ is a group homomorphism (called central charge) such
that
\begin{itemize}
\item $Z$ satisfies the following positivity property for any $E\in \mathcal{A}$:
$$Z(E)\in\{re^{i\pi\phi}: r\geq0, 0<\phi\leq1\}.$$
\item Every
non-zero object in $\mathcal{A}$ has a Harder-Narasimhan filtration
in $\mathcal{A}$ with respect to $\nu_Z$-stability, here the slope
$\nu_Z$ of an object $E\in \mathcal{A}$ is defined by
\begin{eqnarray*}
\nu_{Z}(E)= \left\{
\begin{array}{lcl}
+\infty,  & &\mbox{if}~\Im Z(E)=0,\\
&&\\
-\frac{\Re Z(E)}{\Im Z(E)}, & &\mbox{otherwise}.
\end{array}\right.
\end{eqnarray*}
\end{itemize}
\end{definition}

Let $B$ be a fixed $\mathbb{Q}$-divisor on $X$. Let $\alpha>0$ and
$\beta$ be two real numbers. We will construct a family of weak
Bridgeland stability conditions on $X$ that depends on these two
parameters.


There exists a \emph{torsion pair} $(\mathcal{T}_{\beta
H+B},\mathcal{F}_{\beta H+B})$ in $\Coh(X)$ defined as follows:
\begin{eqnarray*}
\mathcal{T}_{\beta H+B}=\{E\in\Coh(X):\mu^-_{H, \beta H+B}(E)>0 \}\\
\mathcal{F}_{\beta H+B}=\{E\in\Coh(X):\mu^+_{H, \beta H+B}(E)\leq0
\}
\end{eqnarray*}
Equivalently, $\mathcal{T}_{\beta H+B}$ and $\mathcal{F}_{\beta
H+B}$ are the extension-closed subcategories of $\Coh(X)$ generated
by $\mu_{H, \beta H+B}$-stable sheaves of positive and non-positive
slope, respectively.

\begin{definition}
We let $\Coh^{\beta H+B}(X)\subset \D^b(X)$ be the extension-closure
$$\Coh^{\beta H+B}(X)=\langle\mathcal{T}_{\beta H+B}, \mathcal{F}_{
\beta H+B}[1]\rangle.$$
\end{definition}

By the general theory of torsion pairs and tilting \cite{HRS},
$\Coh^{\beta H+B}(X)$ is the heart of a bounded t-structure on
$\D^b(X)$; in particular, it is an abelian category. Consider the
following central charge
$$Z_{\alpha, \beta}(E)=H^{n-2}\Big(\frac{\alpha^2 H^2}{2}\ch_0^{\beta H+B}(E)-\ch_2^{\beta H+B}(E)+i H\ch_1^{\beta H+B}(E)
\Big).$$ We think of it as the composition
$$Z_{\alpha, \beta}: \K(\D^b(X))\xrightarrow{\ch_H} \mathbb{Q}^3 \xrightarrow{z_{\alpha,\beta}}
\mathbb{C},$$ where the first map is given by
$$\ch_H(E)=(H^n\ch^B_0(E), H^{n-1}\ch^B_1(E), H^{n-2}\ch^B_2(E)),$$
and the second map is defined by
\begin{equation}
z_{\alpha, \beta}(e_0, e_1,
e_2)=\frac{1}{2}(\alpha^2-\beta^2)e_0+\beta e_1-e_2+i(e_1-\beta
e_0).
\end{equation}

\begin{theorem}\label{thm2.7}
For any $(\alpha, \beta)\in \mathbb{R}_{>0}\times\mathbb{R}$,
$\sigma_{\alpha, \beta}=(Z_{\alpha, \beta}, \Coh^{\beta H+B}(X))$ is
a weak Bridgeland stability condition.
\end{theorem}
\begin{proof}
The required assertion is proved in \cite{Bri2, AB} for the surface
case. For the threefold case, the conclusion is showed in \cite{BMT,
BMS}. But the proof in \cite[Appendix 2]{BMS} still works for the
general case.
\end{proof}


We write $\nu_{\alpha, \beta}$ for the slope function on $\Coh^{
\beta H+B}(X)$ induced by $Z_{\alpha, \beta}$. Explicitly, for any
$E\in \Coh^{ \beta H+B}(X)$, one has
\begin{eqnarray*}
\nu_{\alpha, \beta}(E)= \left\{
\begin{array}{lcl}
+\infty,  & &\mbox{if}~H^{n-1}\ch^{\beta H+B}_1(E)=0,\\
&&\\
\frac{H^{n-2}\ch_2^{\beta H+B}(E)-\frac{1}{2}\alpha^2H^n\ch^{\beta
H+B}_0(E)}{H^{n-1}\ch^{\beta H+B}_1(E)}, & &\mbox{otherwise}.
\end{array}\right.
\end{eqnarray*}
Theorem \ref{thm2.7} gives the notion of tilt-stability:

\begin{definition}
An object $E\in\Coh^{\beta H+B}(X)$ is \emph{tilt-(semi)stable} (or
$\nu_{\alpha,\beta}$-\emph{(semi)stable}) if, for all non-trivial
subobjects $F\hookrightarrow E$, we have
$$\nu_{\alpha, \beta}(F)<(\leq)\nu_{\alpha, \beta}(E/F).$$
\end{definition}
For any $\mathcal{E}\in\Coh^{\beta H+B}(X)$, the Harder-Narasimhan
property gives a filtration in $\Coh^{\beta H+B}(X)$
$$0=\mathcal{E}_0\subset \mathcal{E}_1\subset\cdots\subset \mathcal{E}_n=\mathcal{E}$$ such that:
$\mathcal{F}_i:=\mathcal{E}_i/\mathcal{E}_{i-1}$ is
$\nu_{\alpha,\beta}$-semistable with
$\nu_{\alpha,\beta}(\mathcal{F}_1)>\cdots>\nu_{\alpha,\beta}(\mathcal{F}_n)$.

\begin{definition}
In the above filtration, we call $\mathcal{F}_1$ the
$\nu_{\alpha,\beta}$-maximal subobject of $\mathcal{E}$ in
$\Coh^{\beta H+B}(X)$, and call $\mathcal{F}_n$ the
$\nu_{\alpha,\beta}$-minimal quotient of $\mathcal{E}$.
\end{definition}

The following well known proposition establishes the relation
between $\nu_{\alpha, \beta}$-stability and $p_{H,B}$-stability.
\begin{proposition}\label{prop1}
For $E\in \Coh^{\beta H+B}(X)\cap\Coh(X)$ and $\alpha\gg0$, we have
$E$ is $\nu_{\alpha, \beta}$-(semi)stable if and only if $E$ is
$p_{H, \beta H+B}$-(semi)stable.
\end{proposition}
\begin{proof}
See \cite[Section 14]{Bri2} and \cite[Appendix A]{Liu}.
\end{proof}

\subsection{Bogomolov-Gieseker type inequality}We now recall the
Bogomolov-Gieseker type inequality for tilt-stable complexes
proposed in \cite{BMT, BMS}.
\begin{definition}
We define the discriminant
$$\Delta:=\ch_1^2-2\ch_0\ch_2=(\ch^B_1)^2-2\ch^B_0\ch^B_2,$$
and the generalized discriminant
$$\overline{\Delta}^{\beta H+B}_H:=(H^{n-1}\ch^{\beta H+B}_1)^2-2H^n\ch^{\beta H+B}_0\cdot(H^{n-2}\ch^{\beta H+B}_2).$$
\end{definition}
A short calculation shows $$\overline{\Delta}^{\beta
H+B}_H=(H^{n-1}\ch_1^B)^2-2H^n\ch_0^B\cdot(H^{n-2}\ch_2^B)=\overline{\Delta}^{B}_H.$$
Hence the generalized discriminant is independent of $\beta$.

\begin{theorem}[Bogomolov, Gieseker]
Assume $E$ is a $\mu_{H}$-semistable torsion free sheaf on $X$. Then
$H^{n-2}\Delta(E)\geq0$.
\end{theorem}
\begin{proof}
See \cite[Theorem 3.4.1]{HL}.
\end{proof}

\begin{theorem}\label{thm2.11}
Assume $E\in\Coh^{\beta H+B}(X)$ is $\nu_{\alpha,\beta}$-semistable,
then $\overline{\Delta}^{B}_H(E)\geq0$.
\end{theorem}
\begin{proof}
This inequality was proved in \cite[Theorem 7.3.1]{BMT} and
\cite[Theorem 3.5]{BMS} on threefolds, but their proof works for the
general case.
\end{proof}

\begin{conjecture}[{\cite[Conjecture 4.1]{BMS}}]\label{Conjecture}
Assume $n=3$, $B=0$ and $E\in\Coh^{\beta H}(X)$ is
$\nu_{\alpha,\beta}$-semistable. Then
\begin{equation}\label{BG}
\alpha^2\overline{\Delta}^{\beta H}_H(E)+4\left(H\ch^{\beta
H}_2(E)\right)^2-6H^2\ch^{\beta H}_1(E)\ch^{\beta H}_3(E)\geq0.
\end{equation}
\end{conjecture}

Such an inequality provides a way to construct Bridgeland stability
conditions on threefolds, and it was proved to hold in the some
cases:
\begin{theorem}\label{*}
The inequality (\ref{BG}) holds for $\nu_{\alpha,\beta}$-semistable
objects on $\mathbb{P}^3$, quadric threefolds, abelian threefolds
and Fano threefolds of Picard number one.
\end{theorem}
\begin{proof}
Please see \cite{Mac2}, \cite{Sch1}, \cite{BMS} and \cite{Li}.
\end{proof}

\begin{remark}
Recently, Schmidt \cite{Sch2} found a counterexample to Conjecture
\ref{Conjecture} when $X$ is the blowup at a point of
$\mathbb{P}^3$. Therefore, the inequality (\ref{BG}) needs some
modifications in general setting. See \cite{Piy} and \cite{BMSZ} for
the recent progress.
\end{remark}

\section{Types of walls}\label{S3}
In this section, we recall some basic properties of walls for the
weak Bridgeland stability $\sigma_{\alpha,\beta}$ in Theorem
\ref{thm2.7}. They are completely analogous to the case of walls for
Bridgeland stability on surfaces, treated most systematically by Lo
and Qin \cite{LQ} and Maciocia \cite{Maci}. We freely use the
notations in Section \ref{S2}.

\subsection{Numerical walls and actual walls}
Let $\mathbf{v}=(v_0, v_1, v_2)$ and $\mathbf{w}=(w_0, w_1, w_2)$ be
two vectors in $\mathbb{Q}^3$ with
\begin{eqnarray*}
\begin{array}{lcl}
\overline{\Delta}_H^B(\mathbf{v})=v_1^2-2v_0v_2\geq0,  &&
\overline{\Delta}_H^B(\mathbf{w})=w_1^2-2w_0w_2\geq0.
\end{array}
\end{eqnarray*}
Assume that $\mathbf{w}$ does not have the same $\nu_{\alpha,
\beta}$-slope as $\mathbf{v}$ everywhere in the $(\beta,
\alpha)$-half plane $\mathbb{R}\times \mathbb{R}_{>0}$. Here the
$\nu_{\alpha, \beta}$-slope of a vector $\mathbf{e}=(e_0, e_1,
e_2)\in\mathbb{Q}^3$ is
\begin{eqnarray*}
\nu_{\alpha, \beta}(\mathbf{e})= \left\{
\begin{array}{lcl}
+\infty,  & &\mbox{if}~e_1-\beta e_0=0,\\
&&\\
\frac{e_2-\beta e_1+\frac{1}{2}(\beta^2-\alpha^2)e_0}{e_1-\beta
e_0}, & &\mbox{otherwise}.
\end{array}\right.
\end{eqnarray*}

\begin{definition}
The numerical wall $W(\mathbf{w}, \mathbf{v})$ is the set of points
$(\beta, \alpha)$ such that $\mathbf{w}$ and $\mathbf{v}$ have the
same $\nu_{\alpha,\beta}$-slope. A numerical wall is an actual wall
if there exists a point $(\beta,\alpha)\in W(\mathbf{w},
\mathbf{v})$ and two objects $E, F\in \Coh^{\beta H+B}(X)$ with
$\ch_H(E)=\mathbf{v}$, $\ch_H(F)=\mathbf{w}$, such that $F$ is a
subobject of $E$, or $E$ is a quotient of $F$ in $\Coh^{\beta
H+B}(X)$. In this situation, we also write $W(F, E)=W(\mathbf{w},
\mathbf{v})$.
\end{definition}

Our definition of an actual wall $W(F, E)$ is different from the one
used earlier (see, e.g., \cite[Definition 2.1]{Bo}). Here we allow
not only that $F$ is a subobject of $E$ but also that $E$ is a
quotient of $F$ in $\Coh^{\beta H+B}(X)$. It makes the related
results more symmetric.

We will frequently use the following facts about the walls
\cite{Maci}, \cite[Section 2.3]{CH1}, \cite[Theorem 2.2 and Section
2.3]{Bo}:
\begin{proposition}\label{wall}
Keep the above notation.
\begin{itemize}
\item The numerical walls $W(\mathbf{w}, \mathbf{v})$ in the $(\beta,\alpha)$-half plane are disjoint.

\item Let $\mathbf{v}$ and $\mathbf{w}$ have positive rank. If $\mu_{H, B}(\mathbf{v})=\mu_{H,
B}(\mathbf{w})$, i.e., $\frac{v_1}{v_0}=\frac{w_1}{w_0}$, then
$W(\mathbf{w}, \mathbf{v})$ is a line $\beta=\mu_{H,
B}(\mathbf{v})$. If $\mu_{H, B}(\mathbf{v})\neq\mu_{H,
B}(\mathbf{w})$, then $W(\mathbf{w}, \mathbf{v})$ is a semicircle
defined by
$(\beta-s(\mathbf{w},\mathbf{v}))^2+\alpha^2=r^2(\mathbf{w},\mathbf{v})$,
where

\begin{equation}\label{s}
s(\mathbf{w},\mathbf{v})=\frac{1}{2}(\mu_{H, B}(\mathbf{v})+\mu_{H,
B}(\mathbf{w}))-\frac{1}{2}
\frac{\overline{\Delta}^B_H(\mathbf{v})/v_0^2-\overline{\Delta}^B_H(\mathbf{w})/w_0^2}{\mu_{H,
B}(\mathbf{v})-\mu_{H, B}(\mathbf{w})},
\end{equation}

\begin{equation}\label{r}
r^2(\mathbf{w},\mathbf{v})=(s(\mathbf{w},\mathbf{v})-\mu_{H,
B}(\mathbf{v}))^2-\overline{\Delta}_H^B(\mathbf{v})/v_0^2.
\end{equation}
When $r^2(\mathbf{w},\mathbf{v})<0$, the wall is empty.

\item Let $W_1$, $W_2$ be two numerical walls to the left of $\beta=\mu_{H, B}(\mathbf{v})$ with centers $(s_1, 0)$, $(s_2,
0)$. Then $W_1$ is nested inside $W_2$ if and only if $s_1> s_2$.

\item Let $\mathbf{v}$ and $\mathbf{w}$ have positive rank and $\mu_{H, B}(\mathbf{v})> \mu_{H, B}(\mathbf{w})$.
If $\mu_{H, B}(\mathbf{w})>\beta$ or $\mu_{H, B}(\mathbf{v})<\beta$,
then $\nu_{\alpha, \beta}(\mathbf{v})>(<)\nu_{\alpha,
\beta}(\mathbf{w})$ if and only if the point $(\beta, \alpha)$ is
outside (inside) $W(\mathbf{w}, \mathbf{v})$. If $\mu_{H,
B}(\mathbf{v})> \beta>\mu_{H, B}(\mathbf{w})$, then $\nu_{\alpha,
\beta}(\mathbf{v})>(<)\nu_{\alpha, \beta}(\mathbf{w})$ if and only
if the point $(\beta, \alpha)$ is inside (outside) $W(\mathbf{w},
\mathbf{v})$.
\end{itemize}
\end{proposition}

Without loss of generality, we now assume $v_0>0$, $w_0>0$, $\mu_{H,
B}(\mathbf{v})> \mu_{H, B}(\mathbf{w})$, and the wall $W(\mathbf{w},
\mathbf{v})$ is a non-empty semicircle.

If $s(\mathbf{w},\mathbf{v})\leq \mu_{H, B}(\mathbf{v})$, i.e.,
\begin{equation}\label{type1+}
\Big(\mu_{H, B}(\mathbf{v})-\mu_{H,B}(\mathbf{w})\Big)^2\geq
\overline{\Delta}^B_H(\mathbf{w})/w^2_0-\overline{\Delta}^B_H(\mathbf{v})/v^2_0,
\end{equation}
since $r(\mathbf{w},\mathbf{v})^2\geq0$, one sees that $\mu_{H,
B}(\mathbf{v})-s(\mathbf{w},\mathbf{v})\geq\sqrt{\overline{\Delta}^B_H(\mathbf{v})}/v_0$.
This and (\ref{s}) imply
\begin{equation*}
\mu_{H, B}(\mathbf{v})-\mu_{H,B}(\mathbf{w})+
\frac{\overline{\Delta}^B_H(\mathbf{v})/v_0^2-\overline{\Delta}^B_H(\mathbf{w})/w_0^2}{\mu_{H,
B}(\mathbf{v})-\mu_{H,
B}(\mathbf{w})}\geq2\sqrt{\overline{\Delta}^B_H(\mathbf{v})}/v_0,
\end{equation*}
i.e.,
$$\Big(\mu_{H, B}(\mathbf{v})-\mu_{H,B}(\mathbf{w})-\sqrt{\overline{\Delta}^B_H(\mathbf{v})}/v_0\Big)^2
\geq \overline{\Delta}^B_H(\mathbf{w})/w_0^2.$$ Hence one obtains
\begin{equation}\label{type1}
\mu_{H, B}(\mathbf{v})-\mu_{H,B}(\mathbf{w})\leq
\sqrt{\overline{\Delta}^B_H(\mathbf{v})}/v_0-\sqrt{\overline{\Delta}^B_H(\mathbf{w})}/w_0
\end{equation}
or
\begin{equation}\label{type2}
\mu_{H, B}(\mathbf{v})-\mu_{H,B}(\mathbf{w})\geq
\sqrt{\overline{\Delta}^B_H(\mathbf{v})}/v_0+\sqrt{\overline{\Delta}^B_H(\mathbf{w})}/w_0
\end{equation}
If $s(\mathbf{w},\mathbf{v})\geq \mu_{H, B}(\mathbf{v})$, a similar
computation gives
\begin{equation}\label{type3}
\mu_{H, B}(\mathbf{v})-\mu_{H,B}(\mathbf{w})\leq
\sqrt{\overline{\Delta}^B_H(\mathbf{w})}/w_0-\sqrt{\overline{\Delta}^B_H(\mathbf{v})}/v_0
\end{equation}

\begin{definition}
The wall $W(\mathbf{w}, \mathbf{v})$ is called of Type 1, if it
satisfies (\ref{type1}). If $W(\mathbf{w}, \mathbf{v})$ satisfies
(\ref{type2}) (respectively, (\ref{type3})), we call it of Type 2
(respectively, Type 3).
\end{definition}

Direct computations show us that the wall of Type 1 lies to the left
of $\beta=\mu_{H, B}(\mathbf{w})<\mu_{H, B}(\mathbf{v})$, the wall
of Type 2 lies between $\beta=\mu_{H, B}(\mathbf{w})$ and
$\beta=\mu_{H, B}(\mathbf{v})$, and the wall of Type 3 is to the
right of $\beta=\mu_{H, B}(\mathbf{v})$ (see Figure \ref{fig1}).

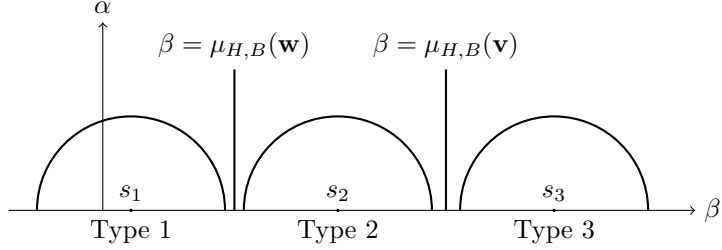
\begin{figure}[ht!]
\begin{center}
\begin{tikzpicture}[scale=1.25]
    \draw[->] (-1,0) -- (6.3,0) node[right] {$\beta$};
    \draw[->] (0,0) -- (0,2) node[above] {$\alpha$};
     \draw[thick] (1.3 cm,0 pt) arc (0:180:1 cm);
     \draw[thick] (3.5 cm,0 pt) arc (0:180:1 cm);
     \draw[thick] (5.8 cm, 0 pt) arc (0: 180: 1 cm);
     \draw[thick] (3.65,0) -- (3.65,1.5) node[above] {$\beta=\mu_{H, B}(\mathbf{v})$};
     \draw[thick] (1.4,0) -- (1.4,1.5) node[above] {$\beta=\mu_{H, B}(\mathbf{w})$};
     \fill (0.3,0) circle (0.5 pt) node[below] {Type 1} node[above] {$s_1$};
     \fill (2.5,0) circle (0.5 pt) node[below] {Type 2} node[above] {$s_2$};
     \fill (4.8,0) circle (0.5 pt) node[below] {Type 3} node[above] {$s_3$};
\end{tikzpicture}
\caption{Three types of walls}\label{fig1}
\end{center}
\end{figure}

\begin{remark}
By the definition of $\Coh^{\beta H+B}(X)$, one sees that an actual
wall can not be of Type 2.
\end{remark}

\subsection{Modifications of walls}
\begin{definition}
Given a vector $\mathbf{u}=(u_0, u_1, u_2)\in \mathbb{Q}^3$ with
$u_0\neq0$, we define its discriminant free vector
$\widetilde{\mathbf{u}}$ to be $(u_0, u_1, u_1^2/2u_0)\in
\mathbb{Q}^3$.
\end{definition}
The motivation behind this definition is that
$\widetilde{\mathbf{u}}$ satisfies
$\overline{\Delta}_H^B(\widetilde{\mathbf{u}})=0$ and $\mu_{H,
B}(\mathbf{u})=\mu_{H, B}(\widetilde{\mathbf{u}})$.

For a numerical wall $W(\mathbf{w}, \mathbf{v})$ of Type 1, we
consider the wall $W(\widetilde{\mathbf{w}}, \mathbf{v})$ to be its
modification. Since
\begin{eqnarray*}
\mu_{H, B}(\mathbf{v})-\mu_{H,B}(\widetilde{\mathbf{w}})&=& \mu_{H,
B}(\mathbf{v})-\mu_{H,B}(\mathbf{w})\\
&\leq&\sqrt{\overline{\Delta}^B_H(\mathbf{v})}/v_0-\sqrt{\overline{\Delta}^B_H(\mathbf{w})}/w_0\\
&\leq&\sqrt{\overline{\Delta}^B_H(\mathbf{v})}/v_0-\sqrt{\overline{\Delta}^B_H(\widetilde{\mathbf{w}})}/w_0,
\end{eqnarray*}
and
\begin{equation*}
\Big(\mu_{H,
B}(\mathbf{v})-\mu_{H,B}(\widetilde{\mathbf{w}})\Big)^2\geq
-\overline{\Delta}^B_H(\mathbf{v})/v^2_0,
\end{equation*}
one sees the wall $W(\widetilde{\mathbf{w}}, \mathbf{v})$ is also of
Type 1. Similarly, if $W(\mathbf{w}, \mathbf{v})$ is of Type 2, then
$W(\widetilde{\mathbf{w}}, \mathbf{v})$, $W(\mathbf{w},
\widetilde{\mathbf{v}})$ and $W(\widetilde{\mathbf{w}},
\widetilde{\mathbf{v}})$ are still of Type 2. If $W(\mathbf{w},
\mathbf{v})$ is of Type 3, then $W(\mathbf{w},
\widetilde{\mathbf{v}})$ is of Type 3.

We can compute the center and radius for the modifications of
$W(\mathbf{w}, \mathbf{v})$ more explicitly. Let $(\widetilde{s}_1,
0)$ and $\widetilde{r}_1$ be the center and radius of the circle
$W(\widetilde{\mathbf{w}}, \mathbf{v})$ of Type 1, respectively.
Equalities (\ref{s}) and (\ref{r}) give
\begin{eqnarray*}
\widetilde{s}_1&=&\frac{1}{2}(\mu_{H, B}(\mathbf{v})+\mu_{H,
B}(\mathbf{w}))-
\frac{\overline{\Delta}^B_H(\mathbf{v})/2v_0^2}{\mu_{H,
B}(\mathbf{v})-\mu_{H, B}(\mathbf{w})},\\
\widetilde{r}_1&=&\frac{\overline{\Delta}^B_H(\mathbf{v})/2v_0^2}{\mu_{H,
B}(\mathbf{v})-\mu_{H, B}(\mathbf{w})}-\frac{1}{2}(\mu_{H,
B}(\mathbf{v})-\mu_{H, B}(\mathbf{w})),\\
\widetilde{s}_1+\widetilde{r}_1&=&\mu_{H, B}(\mathbf{w}),\\
\widetilde{s}_1-\widetilde{r}_1&=&\mu_{H,
B}(\mathbf{v})-\frac{\overline{\Delta}^B_H(\mathbf{v})/v_0^2}{\mu_{H,
B}(\mathbf{v})-\mu_{H, B}(\mathbf{w})}.
\end{eqnarray*}
If $W(\mathbf{w}, \widetilde{\mathbf{v}})$ is of Type 3, we let
$(\widetilde{s}_3, 0)$ and $\widetilde{r}_3$ be its center and
radius, respectively. Similarly, one has
\begin{eqnarray*}
\widetilde{s}_3&=&\frac{1}{2}(\mu_{H, B}(\mathbf{v})+\mu_{H,
B}(\mathbf{w}))+
\frac{\overline{\Delta}^B_H(\mathbf{w})/2w_0^2}{\mu_{H,
B}(\mathbf{v})-\mu_{H, B}(\mathbf{w})},\\
\widetilde{r}_3&=&\frac{\overline{\Delta}^B_H(\mathbf{w})/2w_0^2}{\mu_{H,
B}(\mathbf{v})-\mu_{H, B}(\mathbf{w})}-\frac{1}{2}(\mu_{H,
B}(\mathbf{v})-\mu_{H, B}(\mathbf{w})),\\
\widetilde{s}_3+\widetilde{r}_3&=&\mu_{H,
B}(\mathbf{w})+\frac{\overline{\Delta}^B_H(\mathbf{w})/w_0^2}{\mu_{H,
B}(\mathbf{v})-\mu_{H, B}(\mathbf{w})},\\
\widetilde{s}_3-\widetilde{r}_3&=&\mu_{H, B}(\mathbf{v}).
\end{eqnarray*}
From the above equalities, one sees:
\begin{lemma}
If $W(\mathbf{w}, \mathbf{v})$ is of Type 1 (respectively, 3), then
the semicircle $W(\mathbf{w}, \mathbf{v})$ is inside the semicircle
$W(\widetilde{\mathbf{w}}, \mathbf{v})$ (respectively,
$W(\mathbf{w}, \widetilde{\mathbf{v}})$).
\end{lemma}
A similar conclusion holds for walls of Type 2, but we do not need
it in this paper. See Figure \ref{fig2} for the Modifications of
walls of Type 1 and Type 3. One notices that there is a duality
between the walls of Type 1 and Type 3, namely that $W(\mathbf{w},
\mathbf{v})=W(\mathbf{v}, \mathbf{w})$ and
$W(\widetilde{\mathbf{w}}, \mathbf{v})=W(\mathbf{v},
\widetilde{\mathbf{w}})$. This duality will give rise to the duality
between Theorem \ref{main1} and \ref{main2}.
\begin{figure}[ht!]
\begin{center}
\begin{tikzpicture}[scale=1.25]
    \draw[->] (-1.4,0) -- (5.8,0) node[right] {$\beta$};
    \draw[->] (-0.2,0) -- (-0.2,2.2) node[above] {$\alpha$};
     \draw[thick] (1.3 cm, 0 pt) arc (0:180:1 cm);
     \draw[thick] (1.8 cm, 0 pt) arc (0:180:1.4 cm) ;
\fill (0.4,1.4) circle (0.2 pt) node[above]
{$W(\widetilde{\mathbf{w}}, \mathbf{v})$};

\node[above] at (0.3,0.8) {$W(\mathbf{w}, \mathbf{v})$};

\node[above] at (3.8,0.8) {$W(\mathbf{w}, \mathbf{v})$};

\fill (-1,0) circle (0.7 pt) node[below]
{$\widetilde{s}_1-\widetilde{r}_1$};

\draw[thick] (4.8 cm, 0 pt) arc (0: 180: 1 cm);
     \draw[thick] (5.4 cm, 0 pt) arc (0: 180: 1.4 cm);
\fill (4,1.4) circle (0.2 pt) node[above] {$W(\mathbf{w},
\widetilde{\mathbf{v}})$};

\fill (5.4,0) circle (0.7 pt) node[below]
{$\widetilde{s}_3+\widetilde{r}_3$};
     \draw[thick] (2.6,0) -- (2.6,2) node[above] {$\mu_{H, B}(\mathbf{v})$};
     \draw[thick] (1.8,0) -- (1.8,1.5) node[above] {$\mu_{H, B}(\mathbf{w})$};
     \fill (0.3,0) circle (0.5 pt) node[below] {Type 1} ;
     \fill (3.8,0) circle (0.5 pt) node[below] {Type 3} ;
\end{tikzpicture}
\caption{Modifications of walls}\label{fig2}
\end{center}
\end{figure}
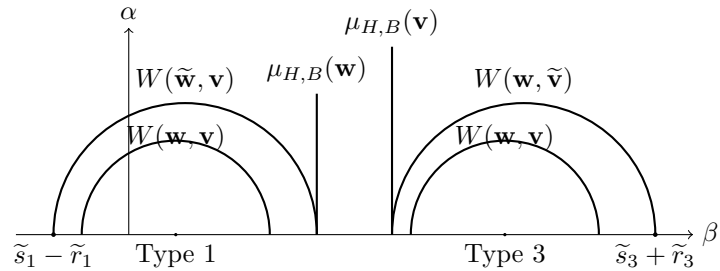

\section{Extremal ellipses}\label{S4}
Throughout this section, we let $E\neq0$ be a torsion free sheaf on
$X$ with $\ch_H(E)=\mathbf{v}=(v_0, v_1, v_2)$ and
$\overline{\Delta}_H^B(\mathbf{v})\geq0$. We will define the
extremal ellipse $C_E$ for such $E$. It can bound the rank of the
subobject or quotient of $E$. We keep the same notations as that in
the previous sections.

\subsection{Extremal ellipses}
The following lemmas are our main tools to study the tilt-stability
of $E$ and $E[1]$. They can be considered as a generalization of
\cite[Lemma 3.1]{Sun}. A similar result on the rank of a
destabilizing subobject is given in \cite[Lemma 7.2]{MS}.

\begin{lemma}\label{lemma4.1}
Let $F$ be the $\nu_{\alpha,\beta}$-maximal subobject of $E\in
\Coh^{\beta H+B}(X)$ for some $(\beta,\alpha)\in \mathbb{R}\times
\mathbb{R}_{>0}$. If
\begin{equation}\label{el}
v_0\big(\beta-\mu_{H,
B}(E)\big)^2+\big(v_0+H^n\big)\alpha^2\geq\frac{v_0+H^n}{v_0H^n}\overline{\Delta}_H^B(E),
\end{equation}
then $\rank(F)\leq\rank(E)$.
\end{lemma}
\begin{proof}
The idea of the proof is rather simple: we firstly combine the
Bogomolov-Gieseker inequality for $\rank F$, $\ch_1(F)$, and
$\ch_2(F)$ and the inequality
$\nu_{\alpha,\beta}(F)>\nu_{\alpha,\beta}(E)$ to give an upper bound
of $\rank F$. Then comparing this bound with $\rank E$, one
concludes the lemma.

Now we show the details. By the long exact sequence of cohomology
sheaves induced by the short exact sequence
$$0\rightarrow F\rightarrow E\rightarrow
Q\rightarrow0$$ in $\Coh^{\beta H+B}(X)$, one sees that $F$ is a
torsion free sheaf. If $E$ is $\nu_{\alpha,\beta}$-semistable, then
$F=E$. Thus we are done.

Now we assume that $E$ is not $\nu_{\alpha,\beta}$-semistable. One
deduces

$$\nu_{\alpha,\beta}(F)=\frac{H^{n-2}\ch^{\beta H+B}_2(F)-\frac{1}{2}\alpha^2H^n\ch_0(F)}{H^{n-1}\ch^{\beta H+B}_1(F)}
>\nu_{\alpha,\beta}(E),$$
i.e.,
\begin{equation}\label{eq4.1}
H^{n-2}\ch^{\beta H+B}_2(F)>\nu_{\alpha,\beta}(E)H^{n-1}\ch^{\beta
H+B}_1(F)+\frac{1}{2}\alpha^2H^n\ch_0(F).
\end{equation}
By Theorem \ref{thm2.11}, we obtain
\begin{equation}\label{eq4.2}
\frac{\left(H^{n-1}\ch^{\beta H+B}_1(F)\right)^2}{2H^n\ch_0(F)}\geq
H^{n-2}\ch^{\beta H+B}_2(F).
\end{equation}
Combining (\ref{eq4.1}) and (\ref{eq4.2}), one sees that
$$\alpha^2\left(H^n\ch_0(F)\right)^2+2\nu_{\alpha,\beta}(E)H^{n-1}\ch^{\beta H+B}_1(F)H^n\ch_0(F)<\left(H^{n-1}\ch^{\beta H+B}_1(F)\right)^2.$$
This implies
\begin{equation}\label{eq4.3}
H^n\ch_0(F)<\left(\sqrt{(\nu_{\alpha,\beta}(E))^2+\alpha^2}-\nu_{\alpha,\beta}(E)\right)\frac{H^{n-1}\ch^{\beta
H+B}_1(F)}{\alpha^2}.
\end{equation}
Since $F$ is a subobject of $E$ in $\Coh^{\beta H+B}(X)$, let $E/F$
be the quotient object in $\Coh^{\beta H+B}(X)$. By the definition
of $\Coh^{\beta H+B}(X)$, we deduce that $\mu^-_{H, \beta
H+B}(F)>0$, $\mu^-_{H, \beta H+B}(E)>0$, $\mu^-_{H, \beta
H+B}(\mathcal{H}^0(E/F))>0$ and $\mu^+_{H, \beta
H+B}(\mathcal{H}^{-1}(E/F))\leq0$. These imply
$$0<H^{n-1}\ch^{\beta H+B}_1(F)\leq H^{n-1}\ch^{\beta H+B}_1(E)=v_1-\beta v_0.$$
From (\ref{eq4.3}), it follows that

\begin{equation}\label{eq4.4}
H^n\ch_0(F)<\left(\sqrt{(\nu_{\alpha,\beta}(E))^2+\alpha^2}-\nu_{\alpha,\beta}(E)\right)\frac{H^{n-1}\ch^{\beta
H+B}_1(E)}{\alpha^2}.
\end{equation}
Hence $\rank(F)\leq\rank(E)$, if one can show that
\begin{equation*}
\left(\sqrt{(\nu_{\alpha,\beta}(E))^2+\alpha^2}-\nu_{\alpha,\beta}(E)\right)\frac{H^{n-1}\ch^{\beta
H+B}_1(E)}{\alpha^2}\leq H^n\ch_0(E)+H^n,
\end{equation*}
i.e.,
\begin{equation}\label{eq4.5}
(\nu_{\alpha,\beta}(E))^2+\alpha^2\leq
\Big(\frac{\alpha^2H^n(\ch_0(E)+1)}{H^{n-1}\ch^{\beta
H+B}_1(E)}+\nu_{\alpha,\beta}(E)\Big)^2.
\end{equation}

On the other hand, a direct computation shows that inequality
(\ref{eq4.5}) is equivalent to
\begin{equation}\label{eq4.6}
(\ch_0(E)+1)H^n\alpha^2+2H^{n-2}\ch_2^{\beta
H+B}(E)\geq\frac{\overline{\Delta}_H^{\beta
H+B}(E)}{H^n}=\frac{\overline{\Delta}_H^{B}(E)}{H^n}.
\end{equation}
Expanding $\ch_2^{\beta H+B}(E)$, one sees that inequality
(\ref{eq4.6}) is equivalent to our assumption (\ref{el}). Thus the
lemma follows.
\end{proof}

The dual result holds for $E[1]$:
\begin{lemma}\label{lemma4.2}
Let $F[1]$ be the $\nu_{\alpha,\beta}$-minimal quotient of $E[1]\in
\Coh^{\beta H+B}(X)$ for some $(\beta,\alpha)\in \mathbb{R}\times
\mathbb{R}_{>0}$. If (\ref{el}) holds, i.e.,
\begin{equation*}
v_0\big(\beta-\mu_{H,
B}(E)\big)^2+\big(v_0+H^n\big)\alpha^2\geq\frac{v_0+H^n}{v_0H^n}\overline{\Delta}_H^B(E),
\end{equation*}
then $\rank(F)\leq\rank(E)$.
\end{lemma}
\begin{proof}
The proof follows in the same way as that of Lemma \ref{lemma4.1}.

We assume that $E[1]$ is not $\nu_{\alpha,\beta}$-semistable. In
this case, one sees that $F$ is a torsion free sheaf with
$$H^{n-1}\ch^{\beta H+B}_1(E)\leq H^{n-1}\ch^{\beta H+B}_1(F)<0$$ and
$\nu_{\alpha,\beta}(E[1])>\nu_{\alpha,\beta}(F[1])$. One can still
obtain (\ref{eq4.1}) and (\ref{eq4.2}). Since $$H^{n-1}\ch^{\beta
H+B}_1(F)<0,$$ (\ref{eq4.3}) becomes
\begin{equation*}
H^n\ch_0(F)<-\left(\sqrt{(\nu_{\alpha,\beta}(E))^2+\alpha^2}+\nu_{\alpha,\beta}(E)\right)\frac{H^{n-1}\ch^{\beta
H+B}_1(F)}{\alpha^2}.
\end{equation*}
This implies
\begin{equation*}
H^n\ch_0(F)<-\left(\sqrt{(\nu_{\alpha,\beta}(E))^2+\alpha^2}+\nu_{\alpha,\beta}(E)\right)\frac{H^{n-1}\ch^{\beta
H+B}_1(E)}{\alpha^2}.
\end{equation*}
Therefore Lemma \ref{lemma4.2} follows, if
\begin{equation*}
-\left(\sqrt{(\nu_{\alpha,\beta}(E))^2+\alpha^2}+\nu_{\alpha,\beta}(E)\right)\frac{H^{n-1}\ch^{\beta
H+B}_1(E)}{\alpha^2}\leq H^n\ch_0(E)+H^n.
\end{equation*}
A direct computation shows that the above inequality is equivalent
to (\ref{el}) in the situation of this lemma. Hence we are done.
\end{proof}

\begin{definition}
We call the curve in the $(\beta, \alpha)$ half plane defined by the
equality of (\ref{el}), i.e.,
\begin{equation}\label{ellipse}
v_0\big(\beta-\mu_{H,
B}(E)\big)^2+\big(v_0+H^n\big)\alpha^2=\frac{v_0+H^n}{v_0H^n}\overline{\Delta}_H^B(E)
\end{equation}
the extremal ellipse of $E$, and denote it by $C_E$.
\end{definition}

\subsection{The intersection of the wall and the extremal ellipse}
Let $\mathbf{w}=(w_0, w_1, w_2)\in\mathbb{Q}^3$ be a vector with
$w_0>0$ and $\overline{\Delta}_H^B(\mathbf{w})=w_1^2-2w_0w_2\geq0$.

\begin{lemma}\label{lemma4.4}
Assume that $\mu_{H, B}(E)>\mu_{H, B}(\mathbf{w})$ and the wall
$W(\mathbf{w},\mathbf{v})$ is of Type 1. Let
$W(\widetilde{\mathbf{w}},\mathbf{v})$ be the modification of
$W(\mathbf{w},\mathbf{v})$.  Then $C_E\cap
W(\widetilde{\mathbf{w}},\mathbf{v})\neq\emptyset$ if and only if
$$\mu_{H, B}(\mathbf{w})>\mu_{H, B}(E)-\frac{1}{H^n\rank E}\sqrt{\frac{\overline{\Delta}_H^B(E)}{\rank E+1}}.$$
\end{lemma}
\begin{proof}
Recall that $W(\widetilde{\mathbf{w}},\mathbf{v})$ is defined by
$(\beta-\widetilde{s}_1)^2+\alpha^2=\widetilde{r}_1^2$, where
\begin{eqnarray}
\widetilde{s}_1&=&\frac{1}{2}(\mu_{H, B}(\mathbf{v})+\mu_{H,
B}(\mathbf{w}))-
\frac{\overline{\Delta}^B_H(\mathbf{v})/2v_0^2}{\mu_{H,
B}(\mathbf{v})-\mu_{H, B}(\mathbf{w})},\label{s1}\\
\widetilde{r}_1&=&\frac{\overline{\Delta}^B_H(\mathbf{v})/2v_0^2}{\mu_{H,
B}(\mathbf{v})-\mu_{H, B}(\mathbf{w})}-\frac{1}{2}(\mu_{H,
B}(\mathbf{v})-\mu_{H, B}(\mathbf{w})).\label{r1}
\end{eqnarray}
Eliminating $\alpha$ from the equations of $C_E$ and
$W(\widetilde{\mathbf{w}},\mathbf{v})$, one obtains
$$(\beta-\widetilde{s}_1)^2-\frac{v_0}{v_0+H^n}\big(\beta-\mu_{H, B}(E)\big)^2=\widetilde{r}_1^2-\frac{\overline{\Delta}_H^B(E)}{v_0H^n},$$
i.e.,
\begin{eqnarray}\label{quad}
\frac{H^n}{v_0+H^n}\beta^2&+&2\Big(\frac{\mu_{H,B}(E)v_0}{v_0+H^n}-\widetilde{s}_1\Big)\beta +\widetilde{s}_1^2\nonumber\\
 &-&\frac{(\mu_{H,B}(E))^2v_0}{v_0+H^n}-
\widetilde{r}_1^2+\frac{\overline{\Delta}_H^B(E)}{v_0H^n}=0.
\end{eqnarray}
We consider (\ref{quad}) to be a quadratic equation with variable
$\beta$. Let $\delta$ be its discriminant. Then one has
\begin{eqnarray*}
\frac{1}{4}\delta&=&\Big(\frac{\mu_{H,B}(E)v_0}{v_0+H^n}-\widetilde{s}_1
\Big)^2-\frac{H^n}{v_0+H^n}\Big(\widetilde{s}_1^2-\frac{(\mu_{H,B}(E))^2v_0}{v_0+H^n}-
\widetilde{r}_1^2+\frac{\overline{\Delta}_H^B(E)}{v_0H^n} \Big)\\
&=&\frac{v_0}{v_0+H^n}\Big((\mu_{H,B}(E))^2+
\widetilde{s}_1^2-2\mu_{H,B}(E)\widetilde{s}_1\Big)+\frac{H^n\widetilde{r}_1^2-\overline{\Delta}_H^B(E)/v_0}{v_0+H^n}\\
&=&\frac{v_0}{v_0+H^n}\Big(\mu_{H,B}(E)-\widetilde{s}_1\Big)^2+\frac{H^n\widetilde{r}_1^2}{v_0+H^n}-
\frac{\overline{\Delta}_H^B(E)/v_0}{v_0+H^n}.
\end{eqnarray*}
Since
$(\mu_{H,B}(E)-\widetilde{s}_1)^2=\widetilde{r}^2_1+\overline{\Delta}_H^B(E)/v^2_0$,
one sees $ \delta=4\widetilde{r}_1^2$. Thus the two solutions of the
quadratic equation (\ref{quad}) are
$$\beta_{\pm}=\frac{v_0+H^n}{H^n}(\widetilde{s}_1\pm\widetilde{r}_1)-\frac{v_0}{H^n}\mu_{H, B}(E).$$
Since the wall $W(\widetilde{\mathbf{w}},\mathbf{v})$ is of Type 1,
one deduces
\begin{eqnarray*}
\beta_{-}-(\widetilde{s}_1-\widetilde{r}_1)=\frac{v_0}{H^n}\big(\widetilde{s}_1-\widetilde{r}_1-\mu_{H,B}(E)\big)<0
\end{eqnarray*}
and
\begin{eqnarray*}
\beta_{+}-(\widetilde{s}_1+\widetilde{r}_1)=\frac{v_0}{H^n}\big(\widetilde{s}_1+\widetilde{r}_1-\mu_{H,B}(E)\big)<0.
\end{eqnarray*}
These imply that $C_E\cap
W(\widetilde{\mathbf{w}},\mathbf{v})\neq\emptyset$ if and only if
$\beta_{+}>\widetilde{s}_1-\widetilde{r}_1$ (see Figure \ref{fig3}).

On the other hand, one has
\begin{eqnarray*}
\beta_{+}-(\widetilde{s}_1-\widetilde{r}_1)&=&\frac{v_0}{H^n}\widetilde{s}_1+\frac{v_0+2H^n}{H^n}\widetilde{r}_1-\frac{v_0}{H^n}\mu_{H,
B}(E)\\
&=&\frac{v_0}{H^n}(\widetilde{s}_1+\widetilde{r}_1-\mu_{H,
B}(E))+2\widetilde{r}_1\\
&=&(\frac{v_0}{H^n}+1)\big(\mu_{H, B}(\mathbf{w})-\mu_{H,
B}(E)\big)+\frac{\overline{\Delta}^B_H(E)/v_0^2}{\mu_{H,
B}(E)-\mu_{H, B}(\mathbf{w})}.
\end{eqnarray*}
Therefore $\beta_{+}>\widetilde{s}_1-\widetilde{r}_1$ if and only if
$$\mu_{H, B}(\mathbf{w})>\mu_{H, B}(E)-\frac{1}{H^n\rank
E}\sqrt{\frac{\overline{\Delta}_H^B(E)}{\rank E+1}}.$$ This
completes the proof.
\end{proof}

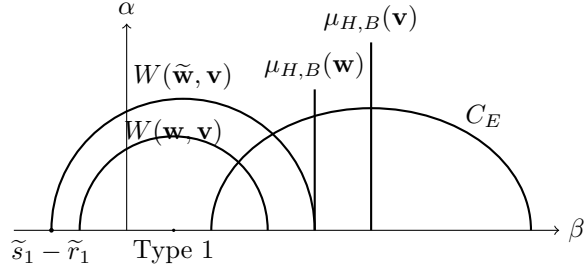
\begin{figure}[ht!]
\begin{center}
\begin{tikzpicture}[scale=1.25]
    \draw[->] (-1.4,0) -- (4.4,0) node[right] {$\beta$};
    \draw[->] (-0.2,0) -- (-0.2,2.2) node[above] {$\alpha$};
     \draw[thick] (1.3 cm, 0 pt) arc (0:180:1 cm);
     \draw[thick] (1.8 cm, 0 pt) arc (0:180:1.4 cm) ;
\fill (0.4,1.4) circle (0.2 pt) node[above]
{$W(\widetilde{\mathbf{w}}, \mathbf{v})$};

\fill (-1,0) circle (0.7 pt) node[below]
{$\widetilde{s}_1-\widetilde{r}_1$};

     \draw[thick] (2.4,0) -- (2.4,2) node[above] {$\mu_{H, B}(\mathbf{v})$};
     \draw[thick] (1.8,0) -- (1.8,1.5) node[above] {$\mu_{H, B}(\mathbf{w})$};
     \fill (0.3,0) circle (0.5 pt) node[below] {Type 1} ;

     \draw[thick] (4.1 cm, 0 pt) arc (0:180:1.7 cm and 1.3 cm);
     \node[above] at (3.6, 1) {$C_E$};
     \node[above] at (0.3, 0.8) {$W(\mathbf{w}, \mathbf{v})$};
\end{tikzpicture}
\caption{Intersection of $W(\widetilde{\mathbf{w}}, \mathbf{v})$ and
$C_E$}\label{fig3}
\end{center}
\end{figure}

For walls of Type 3, we have a similar lemma:
\begin{lemma}\label{lemma4.5}
Assume that $\mu_{H, B}(E)<\mu_{H, B}(\mathbf{w})$ and the wall
$W(\mathbf{v},\mathbf{w})$ is of Type 3. Let
$W(\mathbf{v},\widetilde{\mathbf{w}})$ be the modification of
$W(\mathbf{v},\mathbf{w})$.  Then $C_E\cap
W(\mathbf{v},\widetilde{\mathbf{w}})\neq\emptyset$ if and only if
$$\mu_{H, B}(\mathbf{w})<\mu_{H, B}(E)+\frac{1}{H^n\rank E}\sqrt{\frac{\overline{\Delta}_H^B(E)}{\rank E+1}}.$$
\end{lemma}
\begin{proof}
The proof is the same as that of Lemma \ref{lemma4.4}.
\end{proof}

\section{Tilt-stability of $\mu_{H, B}$-stable sheaves}\label{S5}
The aim of this section is to establish the tilt-stability for a
$\mu_{H, B}$-stable torsion free sheaf via computing the
intersection of the wall and the extremal ellipse. We always assume
that $E$ is a $\mu_{H, B}$-stable torsion free sheaf on $X$ in this
section.





We set
$$\mu^{\max}_{H, B}(E)=\max\Big\{\mu_{H,B}(F):
F~\mbox{is a subsheaf of}~E, \mu_{H,B}(F)\neq\mu_{H,B}(E)\Big\},$$
and let $\mu$ be a rational number satisfying $\mu^{\max}_{H,
B}(E)\leq\mu<\mu_{H,B}(E)$.

\begin{theorem}\label{thm1}
Let $\beta_0=\mu_{H,B}(E)-\frac{\overline{\Delta}_H^B(E)/(H^n\rank
E)^2}{\mu_{H,B}(E)-\mu}$ and
$\beta_1=\mu_{H,B}(E)-\frac{\sqrt{(\rank
E+1)\overline{\Delta}_H^B(E)}}{H^n\rank E}$.
\begin{enumerate}
\item If $\mu>\mu_{H, B}(E)-\frac{1}{H^n\rank
E}\sqrt{\frac{\overline{\Delta}_H^B(E)}{\rank E+1}}$, then $E$ is
$\nu_{\alpha, \beta}$-stable for any $\alpha>0$ and $\beta\leq
\beta_0$.

\item If $\mu\leq\mu_{H, B}(E)-\frac{1}{H^n\rank
E}\sqrt{\frac{\overline{\Delta}_H^B(E)}{\rank E+1}}$ and
$\overline{\Delta}_H^B(E)>0$, then $E$ is $\nu_{\alpha,
\beta_1}$-stable for any $\alpha>0$.

\item If $\overline{\Delta}_H^B(E)=0$, then $E$ is $\nu_{\alpha,
\beta}$-stable for any $\alpha>0$ and $\beta<\mu_{H,B}(E)$.
\end{enumerate}
\end{theorem}
\begin{proof}
(1) We prove the first statement firstly. Choose a vector
$\mathbf{u}=(u_0, u_1, u_2)\in\mathbb{Q}^3$ such that $u_0>0$,
$\frac{u_1}{u_0}=\mu$ and $u_1^2-2u_0u_2=0$. By (\ref{s}) and
(\ref{r}), one sees that the left intersection point of
$W(\mathbf{u},\mathbf{v})$ and $\beta$-axis is $(\beta_0, 0)$. We
assume $\alpha>0$ and $\beta\leq \beta_0$. The condition
$$\mu>\mu_{H, B}(E)-\frac{1}{H^n\rank
E}\sqrt{\frac{\overline{\Delta}_H^B(E)}{\rank E+1}}$$ implies that
\begin{eqnarray*}
\beta_0&=&\mu_{H,B}(E)-\frac{\overline{\Delta}_H^B(E)/(H^n\rank
E)^2}{\mu_{H,B}(E)-\mu}\\
&<&\mu_{H,B}(E)-\frac{1}{H^n\rank E}\sqrt{(\rank
E+1)\overline{\Delta}_H^B(E)}=\beta_1<\mu_{H,B}(E).
\end{eqnarray*}
Since the left intersection point of $C_E$ and $\beta$-axis is just
$$\Big(\mu_{H,B}(E)-\frac{1}{H^n\rank E}\sqrt{(\rank
E+1)\overline{\Delta}_H^B(E)}, 0\Big)=(\beta_1, 0),$$ the point
$(\beta_0, \alpha)$ is outside the ellipse $C_E$. By Lemma
\ref{lemma4.1}, the $\nu_{\alpha,\beta}$-maximal subobject $F$ of
$E\in\Coh^{\beta H+B}(X)$ satisfies $\rank F\leq \rank E$. Set
$\ch_H(F)=\mathbf{w}$.

\medskip
\noindent{\bf Step 1.} {\em $\nu_{\alpha,\beta}$-semistability of
$E$. }

From the definition of $\Coh^{\beta H+B}(X)$, one sees that
$\beta<\mu_{H,B}(F)$. If $\mu_{H,B}(F)\leq \mu$, one can assume that
$W(\mathbf{w}, \mathbf{v})$ is an actual wall of Type 1. We have
$W(\mathbf{w}, \mathbf{v})$ is inside its modification
$W(\widetilde{\mathbf{w}}, \mathbf{v})$. By Proposition \ref{wall},
one sees $W(\widetilde{\mathbf{w}}, \mathbf{v})$ is inside the wall
$W(\mathbf{u},\mathbf{v})$. It follows that the point $(\beta,
\alpha)$ is outside the wall $W(\mathbf{w}, \mathbf{v})$. Therefore
we conclude that $\nu_{\alpha,\beta}(F)<\nu_{\alpha,\beta}(E)$. This
contradicts our assumption that $F$ is the
$\nu_{\alpha,\beta}$-maximal subobject of $E\in\Coh^{\beta H+B}(X)$.

Hence we obtain $\mu_{H,B}(F)>\mu\geq\mu^{\max}_{H, B}(E)$.
Considering the corresponding exact sequence
$$0\rightarrow F\rightarrow E\rightarrow Q\rightarrow0,$$ in
$\Coh^{\beta H+B}(X)$, we get a long exact sequence in $\Coh(X)$:
$$0\rightarrow \mathcal{H}^{-1}(Q)\rightarrow F\rightarrow E\rightarrow \mathcal{H}^0(Q)\rightarrow 0.$$
When $\mathcal{H}^{-1}(Q)\neq0$, one has
$\mu_{H,B}(\mathcal{H}^{-1}(Q))\leq0<\mu_{H,B}(F)$ and
$\rank(F/\mathcal{H}^{-1}(Q))<\rank E$. This implies
\begin{equation*}
\mu_{H,B}(F)<\mu_{H,B}(F/\mathcal{H}^{-1}(Q))\leq \mu^{\max}_{H,
B}(E)\leq\mu,
\end{equation*}
which is absurd. Thus $\mathcal{H}^{-1}(Q)=0$, and $F$ is a subsheaf
of $E$. By the definition of $\mu^{\max}_{H, B}(E)$, one sees that
$\mu_{H, B}(F)=\mu_{H, B}(E)$ and $\rank F=\rank E$. Hence $Q$ is a
torsion sheaf, and the codimension of the support of $Q$ is $\geq2$.
It follows that $\nu_{\alpha_0,\beta_0}(F)\leq
\nu_{\alpha,\beta}(E)$. Therefore we conclude that $E$ is
$\nu_{\alpha,\beta}$-semistable.

\medskip
\noindent{\bf Step 2.} {\em $\nu_{\alpha,\beta}$-stability of $E$. }

We argue by contradiction to show the $\nu_{\alpha,\beta}$-stability
of $E$. Suppose that there is a subobject $K\subset E$ in
$\Coh^{\beta H+B}(X)$ such that $\nu_{\alpha,
\beta}(K)\geq\nu_{\alpha, \beta}(E/K)$. One sees that $\nu_{\alpha,
\beta}(K)=\nu_{\alpha, \beta}(E)$, $0<\ch_1^{\beta
H+B}(K)<\ch_1^{\beta H+B}(E)$, and $K$ is
$\nu_{\alpha,\beta}$-semistable. This implies that (\ref{eq4.4}) in
the proof of Lemma \ref{lemma4.1} holds for $K$. Hence Lemma
\ref{lemma4.1} holds also for $K$. The proof of Step 1 shows us that
$K$ is a subsheaf of $E$, $\mu_{H, B}(K)=\mu_{H, B}(E)$ and $\rank
K=\rank E$. It contradicts that $\ch_1^{\beta H+B}(K)<\ch_1^{\beta
H+B}(E)$. Thus we conclude that $E$ is $\nu_{\alpha,\beta}$-stable.

\bigskip
(2) Now we prove the second statement in the same way as
above. We assume $\alpha>0$.

The assumption $\overline{\Delta}^B_H(E)>0$ makes sure that $E$ is
an object in $\Coh^{\beta_1 H+B}(X)$. Since the point $(\beta_1,
\alpha)$ is outside the ellipse $C_E$, by Lemma \ref{lemma4.1}, the
$\nu_{\alpha,\beta_1}$-maximal subobject $F'$ of $E$ in
$\Coh^{\beta_1 H+B}(X)$ satisfies $\rank F'\leq \rank E$. Set
$\ch_H(F')=\mathbf{w}'$. As in Step 1, if $\mu_{H,B}(F')\leq \mu$,
we can assume that $W(\mathbf{w}', \mathbf{v})$ is an actual wall of
Type 1. Hence $W(\mathbf{w}', \mathbf{v})$ is inside its
modification $W(\widetilde{\mathbf{w}'}, \mathbf{v})$. By Lemma
\ref{lemma4.4} and the condition
$$\mu\leq\mu_{H, B}(E)-\frac{1}{H^n\rank
E}\sqrt{\frac{\overline{\Delta}_H^B(E)}{\rank E+1}},$$ it follows
that $W(\widetilde{\mathbf{w}'}, \mathbf{v})\cap C_E=\emptyset$. On
the other hand, from the definition of $\Coh^{\beta_1 H+B}(X)$, one
sees $\beta_1<\mu_{H,B}(F')$. Hence $W(\widetilde{\mathbf{w}'},
\mathbf{v})$ is inside $C_E$. This implies that $(\beta_1, \alpha)$
is outside the wall $W(\mathbf{w}',\mathbf{v})$. Thus
$\nu_{\alpha,\beta_1}(F)<\nu_{\alpha,\beta_1}(E)$ which is absurd.
When $\mu_{H,B}(F')>\mu\geq \mu^{\max}_{H, B}(E)$, the proof in Step
1 still works here. It turns out that $E$ is $\nu_{\alpha,
\beta_1}$-semistable. The same argument in Step 2 shows that $E$ is
$\nu_{\alpha, \beta_1}$-stable.

\bigskip
(3) To show the third
statement, one can just replace $\beta_1$ in the proof of the second
statement to any $\beta<\mu_{H,B}(E)$.
\end{proof}

\begin{remark}\label{remark}
The above theorem can be improved when $\rank E=1$. In that case, if
$(\beta,\alpha)$ is outside $C_E$, the $\nu_{\alpha,\beta}$-maximal
subobject of $E$ is just a subsheaf of $E$. Hence one can consider
the wall $W(G,E)$ which satisfies
\begin{enumerate}
\item $G$ is a subsheaf of $E$.
\item $W(G,E)$ is of Type 1.
\item $W(G,E)$ is as large as possible, subject to (1) and (2).
\end{enumerate}
Computing the intersection of $C_E$ and $W(G,E)$, one can obtain a
better result.
\end{remark}

Now we consider the tilt-stability of $E[1]$. Set
$$\mathcal{Q}=\big\{Q\in\Coh(X): Q~\mbox{is a torsion free quotient
sheaf of}~E\big\}$$ and
$$\mathcal{T}=\big\{G\in\Coh(X): G~\mbox{is a torsion free extension
of a torsion sheaf by}~E\big\}.$$ One sees that any $Q\in
\mathcal{Q}\cup\mathcal{T}$ satisfies $\mu_{H,B}(Q)\geq
\mu_{H,B}(E)$.

We define
$$\mu^{\min}_{H, B}(E)=\min\big\{\mu_{H,B}(Q): Q\in
\mathcal{Q}\cup\mathcal{T}~ \mbox{and}~
\mu_{H,B}(Q)\neq\mu_{H,B}(E)\big\}.$$

\begin{remark}
The definition of $\mu^{\min}_{H, B}(E)$ is not like that of
$\mu^{\max}_{H, B}(E)$. We must consider the sheaf in $\mathcal{T}$
here, when we investigate the tilt-stability of $E[1]$. The reason
is that a torsion free sheaf has no torsion subsheaf, but can have a
torsion quotient.
\end{remark}

\begin{theorem}\label{thm2}
Assume that $E$ is a $\mu_{H,B}$-stable reflexive sheaf, and
$\bar{\mu}$ is a rational number satisfies $\mu_{H,
B}(E)<\bar{\mu}\leq\mu^{\min}_{H,B}(E)$. Let
$\overline{\beta}_0=\mu_{H,B}(E)+\frac{\overline{\Delta}_H^B(E)/(H^n\rank
E)^2}{\bar{\mu}-\mu_{H, B}(E)}$ and
$\overline{\beta}_1=\mu_{H,B}(E)+\frac{\sqrt{(\rank
E+1)\overline{\Delta}_H^B(E)}}{H^n\rank E}$.
\begin{enumerate}
\item If $\bar{\mu}<\mu_{H, B}(E)+\frac{1}{H^n\rank
E}\sqrt{\frac{\overline{\Delta}_H^B(E)}{\rank E+1}}$, then $E[1]$ is
$\nu_{\alpha, \beta}$-stable for any $\alpha>0$ and
$\beta\geq\overline{\beta}_0$.

\item If $\bar{\mu}\geq\mu_{H, B}(E)+\frac{1}{H^n\rank
E}\sqrt{\frac{\overline{\Delta}_H^B(E)}{\rank E+1}}$, then $E[1]$ is
$\nu_{\alpha, \overline{\beta}_1}$-stable for any $\alpha>0$.

\item If $\overline{\Delta}_H^B(E)=0$, then $E[1]$ is
$\nu_{\alpha,\beta}$-stable for any $\alpha>0$ and $\beta\geq\mu_{H,
B}(E)$.
\end{enumerate}
\end{theorem}
\begin{proof}
The proof is almost the same as that of Theorem \ref{thm1}. The only
difference is that in our situation, when we consider the
$\nu_{\alpha,\beta}$-minimal quotient $Q[1]$ of $E[1]$ with
$\mu_{H,B}(Q)<\bar{\mu}\leq\mu^{\min}_{H,B}(E)$, we have
$\mu_{H,B}(E)=\mu_{H,B}(Q)$ and $\rank E=\rank Q$. One gets a short
exact sequence in $\Coh^{\beta H+B}(X)$:
$$0\rightarrow T\rightarrow E[1]\rightarrow Q[1]\rightarrow0,$$
where $T$ is a torsion sheaf with $\codim(\Supp(T))\geq2$. When
$T\neq0$, one sees that $\nu_{\alpha,\beta}(T)=+\infty$, hence
$E[1]$ can not be $\nu_{\alpha,\beta}$-stable. In order to exclude
this case, we need the reflexive assumption of $E$. To see this, one
notices that $T$ is a subsheaf of $E^{**}/E$ and so the result is
obvious.
\end{proof}

One may note that there is a duality between Theorem \ref{thm1} and
\ref{thm2}. Indeed let $F$ be a subsheaf of $E$ such that $E/F$ is
also torsion free. It induces an exact sequence $$0\rightarrow
(E/F)^*\rightarrow E^*\xrightarrow{\varphi}F^*.$$ The torsion
freeness guarantees that $\varphi$ is surjective on an open subset
$U$ of $X$ with $\codim(X-U)\geq2$. Hence
$c_1(F^*)=c_1(\varphi(E^*))=-c_1(F)$. On the other hand, if $F$ is a
subsheaf of $E$ with $\rank F=\rank E$, then $(E/F)^*=0$. One sees
that $F^*$ is an extension of a torsion sheaf by $E^*$. Thus
$\mu^{\min}_{H,B}(E^*)\leq-\mu^{\max}_{H,B}(E)$. Conversely, it
turns out that there is an open subset $V$ of $X$ with
$\codim(X-V)\geq2$ such that $E|_V=E^{**}|_V$. Thus from a torsion
free quotient of $E^*$ or a torsion free extension of a torsion
sheaf by $E^*$, one can dually get a subsheaf of $E$. This implies
$\mu^{\max}_{H,B}(E)\geq-\mu^{\min}_{H,B}(E^*)$. Therefore, we
conclude that $\mu^{\max}_{H,B}(E)=-\mu^{\min}_{H,B}(E^*)$.

After replacing $E$ (resp. $\bar{\mu}$ and $\beta$) in Theorem
\ref{thm2} by $E^*$ (resp. $-\mu$ and $-\beta$), one finds that
$\overline{\beta}_0=-\beta_0$, $\overline{\beta}_1=-\beta_1$, and
Theorem \ref{thm1} and \ref{thm2} have totally the same form.

\section{Vanishing theorem for $\mu_{H,B}$-stable sheaves}\label{S6}
In this section we prove Corollary \ref{main3}, Corollary
\ref{main4}, Theorem \ref{Serre} and Corollary \ref{corv}. Corollary
\ref{main3} and Corollary \ref{main4} follow from the lemmas below.
We always assume $B=0$ in this section.

\begin{lemma}\label{lemma5.1}
Let $\mathcal{F}$ be an object in $\Coh^{\beta H}(X)$. If
$\mathcal{F}$ is $\nu_{\alpha,\beta}$-stable for any $\alpha>0$,
then $H^{n-1}(X, \mathcal{F}(K_S+lH))=0$ for any integer $l>-\beta$.
\end{lemma}
\begin{proof}
Serre duality implies $$H^{n-1}(X, \mathcal{F}(K_S+lH))^{\vee}\cong
\Ext^1(\mathcal{F}, \mathcal{O}_X(-lH))\cong\Hom(\mathcal{F},
\mathcal{O}_X(-lH)[1]).$$ Since
$\overline{\Delta}_H(\mathcal{O}_X(-lH))=0$, by Theorem \ref{thm2},
one sees that $\mathcal{O}_X(-lH)[1]$ is $\nu_{\alpha,\beta}$-stable
for any $\alpha>0$ and $l\geq-\beta$.

On the other hand, the wall $W(\mathcal{F}, \mathcal{O}_X(-lH)[1])$
is of Type 2, and is defined by $(\beta-s)^2+\alpha^2=r^2$, where
$s-r=-l$. Hence the point $(\beta, 0)$ is inside the wall
$W(\mathcal{F}, \mathcal{O}_X(-lH)[1])$ if $l>-\beta$. This implies
that
$\nu_{\alpha,\beta}(\mathcal{F})>\nu_{\alpha,\beta}(\mathcal{O}_X(-lH)[1])$
for $l>-\beta$ and some $\alpha>0$. From the
$\nu_{\alpha,\beta}$-stability of $\mathcal{F}$, it follows that
$\Hom(\mathcal{F}, \mathcal{O}_X(-lH)[1])=0$.
\end{proof}

\begin{lemma}\label{lemma5.2}
Let $\mathcal{F}[1]$ be an object in $\Coh^{\beta H}(X)$. If
$\mathcal{F}[1]$ is $\nu_{\alpha,\beta}$-stable for any $\alpha>0$,
then $H^1(X, \mathcal{F}(-lH))=0$ for any integer $l>\beta$.
\end{lemma}
\begin{proof}
We consider the Type 2 wall $W(\mathcal{O}_X(lH), \mathcal{F}[1])$:
$(\beta-s)^2+\alpha^2=r^2$, where $s+r=l$ and
$\mu_{H,B}(\mathcal{F})\leq \beta<l$. Hence the point $(\beta, 0)$
is inside the wall $W(\mathcal{O}_X(lH), \mathcal{F}[1])$, if
$l>\beta$. It follows that
$\nu_{\alpha,\beta}(\mathcal{O}_X(lH))>\nu_{\alpha,\beta}(\mathcal{F}[1])$
for some $\alpha>0$ and $l>\beta$. By the
$\nu_{\alpha,\beta}$-stability of $\mathcal{O}_X(lH)$ and
$\mathcal{F}[1]$, we get our conclusion.
\end{proof}

\begin{proof}[Proof of Corollary \ref{main3} and \ref{main4}]
Corollary \ref{main3} follows from Theorem \ref{thm1} and Lemma
\ref{lemma5.1}, and Corollary \ref{main4} follows from Theorem
\ref{thm2} and Lemma \ref{lemma5.2},
\end{proof}

Corollary \ref{main3} and \ref{main4} can give an effective Serre
vanishing theorem for $H^{n-1}$. The bound will be stated as the
function in Definition\ref{round}, and it is more precise than the
bound in Theorem \ref{Serre}.

\begin{theorem}\label{Se}
Let $\mathcal{F}$ be a coherent torsion free sheaf on $X$, and let
$0=\mathcal{F}_0\subset \mathcal{F}_1\subset\cdots\subset
\mathcal{F}_k=\mathcal{F}$ be its Harder-Narasimhan filtration. Set
$\mathcal{G}_i=\mathcal{F}_i/\mathcal{F}_{i-1}$. Then $H^{n-1}(X,
\mathcal{F}(K_X+lH))=0$ as soon as
\begin{eqnarray*}
l>M(\mathcal{F}):=&&\max_{1\leq i\leq
k}\Big\{\frac{\overline{\Delta}_H(\mathcal{G}_i)/(H^n\rank
\mathcal{G}_i)}{H^n\mu_{H}(\mathcal{G}_i)-[H^n\mu_{H}(\mathcal{G}_i)]_{\rank
\mathcal{G}_i}}-\mu_{H}(\mathcal{G}_i),\\
&&\sqrt{\frac{2\overline{\Delta}_H(\mathcal{G}_i)}{(H^n)^2\rank
\mathcal{G}_i}}-\mu_{H}(\mathcal{G}_i)\Big\}.
\end{eqnarray*}
\end{theorem}
\begin{proof}
Since $h^{n-1}(X, \mathcal{F}(K_X+lH))\leq\sum^k_{i=1} h^{n-1}(X,
\mathcal{G}_i(K_X+lH))$, one can assume that $\mathcal{F}$ is
$\mu_H$-semistable. Consider its Jordan-H\"older filtration
$$0=\mathcal{F}_0'\subset\mathcal{F}_1'\subset\cdots\subset\mathcal{F}_{m-1}'\subset\mathcal{F}_m'=\mathcal{F},$$
and set $\mathcal{Q}_i$ be the $\mu_H$-stable sheaf
$\mathcal{F}_{i}'/\mathcal{F}_{i-1}'$. It turns out that
$\mu_{H}(\mathcal{Q}_i)=\mu_H(\mathcal{F})$, for any $i>0$.

For any torsion free sheaf $E$ on $X$, one sees that
$H^n\mu_H^{\max}(E)\leq[H^n\mu_H(E)]_{\rank E}$. Thus applying
Corollary \ref{main3} for $E=\mathcal{Q}_i$ and
$\mu=[H^n\mu_H(\mathcal{Q}_i)]_{\rank \mathcal{Q}_i}/H^n$, one
deduces that $h^{n-1}(X, \mathcal{F}(K_X+lH))\leq\sum^k_{i=1}
h^{n-1}(X, \mathcal{Q}_i(K_X+lH))$=0 for
\begin{eqnarray*}
l>&&\max_{1\leq i\leq
m}\Big\{\frac{\overline{\Delta}_H(\mathcal{Q}_i)/H^n(\rank
\mathcal{Q}_i)^2}{H^n\mu_{H}(\mathcal{Q}_i)-[H^n\mu_{H}(\mathcal{Q}_i)
]_{\rank \mathcal{Q}_i}}-\mu_{H}(\mathcal{Q}_i),\\
&&\frac{\sqrt{(\rank
\mathcal{Q}_i+1)\overline{\Delta}_H(\mathcal{Q}_i)}}{H^n\rank
\mathcal{Q}_i}-\mu_{H}(\mathcal{Q}_i)\Big\}.
\end{eqnarray*}
From
\begin{eqnarray*}
\frac{\overline{\Delta}_H(\mathcal{F})}{H^n\rank \mathcal{F}}&=&
\mu_{H}(\mathcal{F})H^{n-1}\ch_1(\mathcal{F})-2H^{n-2}\ch_2(\mathcal{F})\\
&=&\mu_{H}(\mathcal{F})\sum_{i=1}^mH^{n-1}\ch_1(\mathcal{Q}_i)-2\sum_{i=1}^mH^{n-2}\ch_2(\mathcal{Q}_i)\\
&=&\sum_{i=1}^m\Big(\mu_{H}(\mathcal{Q}_i)H^{n-1}\ch_1(\mathcal{Q}_i)-2H^{n-2}\ch_2(\mathcal{Q}_i)\Big)\\
&=&\sum_{i=1}^m\frac{\overline{\Delta}_H(\mathcal{Q}_i)}{H^n\rank
\mathcal{Q}_i},
\end{eqnarray*}
it follows that $\frac{\overline{\Delta}_H(\mathcal{Q}_i)}{H^n\rank
\mathcal{Q}_i}\leq\frac{\overline{\Delta}_H(\mathcal{F})}{H^n\rank
\mathcal{F}}$. Hence one sees that
\begin{equation}\label{eq6.1}
\frac{\overline{\Delta}_H(\mathcal{Q}_i)/H^n(\rank
\mathcal{Q}_i)^2}{H^n\mu_{H}(\mathcal{Q}_i)-[H^n\mu_{H}(\mathcal{Q}_i)
]_{\rank \mathcal{Q}_i}}\leq
\frac{\overline{\Delta}_H(\mathcal{F})/(H^n\rank
\mathcal{F})}{H^n\mu_{H}(\mathcal{F})-[H^n\mu_{H}(\mathcal{F})]_{\rank
\mathcal{F}}}
\end{equation}
and $$\frac{\sqrt{(\rank
\mathcal{Q}_i+1)\overline{\Delta}_H(\mathcal{Q}_i)}}{H^n\rank
\mathcal{Q}_i}-\mu_{H}(\mathcal{Q}_i)\leq
\frac{1}{H^n}\sqrt{\frac{2\overline{\Delta}_H(\mathcal{F})}{\rank
\mathcal{F}}}-\mu_{H}(\mathcal{F}).$$ It turns out that $H^{n-1}(X,
\mathcal{F}(K_X+lH))=0$, if
\begin{eqnarray*}
l>&&\max\Big\{\frac{\overline{\Delta}_H(\mathcal{F})/(H^n\rank
\mathcal{F})}{H^n\mu_{H}(\mathcal{F})-[H^n\mu_{H}(\mathcal{F})]_{\rank
\mathcal{F}}}-\mu_{H}(\mathcal{F}),\\
&&\sqrt{\frac{2\overline{\Delta}_H(\mathcal{F})}{(H^n)^2\rank
\mathcal{F}}}-\mu_{H}(\mathcal{F})\Big\}.
\end{eqnarray*}
This completes the proof.
\end{proof}

\begin{proof}[Proof of Theorem \ref{Serre}]
The constant $M(\mathcal{F})$ in Theorem \ref{Se} can have a simpler
but weaker form as stated in Theorem \ref{Serre}. In fact, by the
following lemma, one sees that (\ref{eq6.1}) in the proof above
becomes
\begin{equation*}
\frac{\overline{\Delta}_H(\mathcal{Q}_i)/H^n(\rank
\mathcal{Q}_i)^2}{H^n\mu_{H}(\mathcal{Q}_i)-[H^n\mu_{H}(\mathcal{Q}_i)
]_{\rank \mathcal{Q}_i}}\leq
\frac{\overline{\Delta}_H(\mathcal{Q}_i)}{H^n}\leq\frac{\overline{\Delta}_H(\mathcal{F})}{H^n}.
\end{equation*}
It follows that $$M(\mathcal{F})\leq\max_{1\leq i\leq
k}\Big\{\frac{\overline{\Delta}_H(\mathcal{G}_i)}{H^n}-\mu_{H}(\mathcal{G}_i),
\sqrt{\frac{2\overline{\Delta}_H(\mathcal{G}_i)}{(H^n)^2\rank
\mathcal{G}_i}}-\mu_{H}(\mathcal{G}_i)\Big\}.$$ This implies the
desired conclusion.
\end{proof}

\begin{lemma}\label{lem}
Let $r\geq1$ and $d$ be two integers. Then $[\frac{d}{r}]_r\leq
\frac{d}{r}-\frac{1}{r^2}$.
\end{lemma}
\begin{proof}
We argue by contradiction. Assume that there are two integers $a$
and $b$ satisfy $\frac{a}{b}<\frac{d}{r}$, $1\leq b\leq r$ and
$\frac{a}{b}>\frac{d}{r}-\frac{1}{r^2}$. Then one sees that
$$ar>bd-\frac{b}{r}\geq bd-1.$$ This implies $ar\geq bd$. Hence
$\frac{a}{b}\geq\frac{d}{r}$ which is absurd.
\end{proof}

\begin{proof}[Proof of Corollary \ref{corv}]
By Theorem \ref{Serre}, one sees that $H^{n-1}(X,
\mathcal{F}(K_X+lH))=0$, if
$$l>\max\Big\{\frac{\overline{\Delta}_H(\mathcal{F})}{H^n}-\mu_{H}(\mathcal{F}),
\sqrt{\frac{2\overline{\Delta}_H(\mathcal{F})}{(H^n)^2\rank
\mathcal{F}}}-\mu_{H}(\mathcal{F})\Big\}.$$

From $\rank \mathcal{F}\geq2$, it follows
$$\frac{\overline{\Delta}_H(\mathcal{F})}{H^n}-\mu_{H}(\mathcal{F})\geq\sqrt{\frac{2\overline{\Delta}_H(\mathcal{F})}{(H^n)^2\rank
\mathcal{F}}}-\mu_{H}(\mathcal{F}).$$ Thus we are done.
\end{proof}

\section{Chern classes of $\mu_{H}$-stable sheaves on $\mathbb{P}^3$}\label{S7}
In this section we exhibit the applications of Theorem \ref{main1}
to the Chern classes of $\mu_{H,B}$-stable sheaves on
$\mathbb{P}^3$, and prove Theorem \ref{Chern}. From now on, we
assume that $B=0$, $X=\mathbb{P}^3$, $H$ is a plane on
$\mathbb{P}^3$, and $E$ is a $\mu_H$-stable torsion free sheaf on
$\mathbb{P}^3$.

\begin{proof}[Proof of Theorem \ref{Chern}]
If $E$ is $\nu_{\alpha,\beta}$-semistable for any $\alpha>0$, by
Theorem \ref{*}, we have
\begin{equation*}
\frac{\alpha^2}{6}\overline{\Delta}^{\beta
H}_H(E)+\frac{2}{3}\left(\ch^{\beta H}_2(E)\right)^2\geq \ch^{\beta
H}_1(E)\ch^{\beta H}_3(E).
\end{equation*}
By setting $\alpha\rightarrow 0$, one sees
\begin{equation}\label{eq7.1}
\frac{2}{3}\left(\ch^{\beta H}_2(E)\right)^2\geq \ch^{\beta
H}_1(E)\ch^{\beta H}_3(E).
\end{equation}
Substituting
\begin{eqnarray*}
\ch^{\beta H}_1(E)&=&\ch_1(E)-\beta \rank E\\
\ch^{\beta H}_2(E)&=&\ch_2(E)-\beta \ch_1(E)+\frac{\beta^2}{2}\rank
E\\
\ch^{\beta H}_3(E)&=&\ch_3(E)-\beta
\ch_2(E)+\frac{\beta^2}{2}\ch_1(E)-\frac{\beta^3}{6}\rank E
\end{eqnarray*}
into (\ref{eq7.1}), we have
$$4\ch^2_2(E)+\beta^2\Delta(E)-2\beta\ch_1(E)\ch_2(E)\geq6\rank E\big(\mu_{H}(E)-\beta\big)\ch_3(E).$$
This implies
\begin{equation}\label{eq7.2}
\frac{\Delta(E)}{\rank E}\Big(\mu_{H}(E)-\beta
+\frac{\Delta(E)/(\rank E)^2}{\mu_{H}(E)-\beta
}\Big)+6l(E)\geq6\ch_3(E),
\end{equation}
here $l(E)=\frac{c_1^3(E)-3c_1(E)\Delta(E)}{6(\rank E)^2}$.

From Theorem \ref{main1}, one sees that if
$\mu^{\max}_{H}(E)>\mu_{H}(E)-\frac{1}{\rank
E}\sqrt{\frac{\Delta(E)}{\rank E+1}}$, then $E$ is
$\nu_{\alpha,\beta_0}$-stable for any $\alpha>0$ and
$\beta_0=\mu_H(E)-\frac{\Delta(E)/(\rank
E)^2}{\mu_H(E)-\mu^{\max}_H(E)}$. Hence

\begin{equation*}
\frac{\Delta(E)}{\rank
E}\Big(\mu_{H}(E)-\mu^{\max}_{H}(E)+\frac{\Delta(E)/(\rank
E)^2}{\mu_{H}(E)-\mu^{\max}_{H}(E)}\Big)+6l(E)\geq6\ch_3(E).
\end{equation*}
Since $[\mu_{H}(E)]_{\rank E}\geq\mu^{\max}_{H}(E)$, we deduce
\begin{equation*} \frac{\Delta(E)}{6\rank
E}\Big(\mu_{H}(E)-[\mu_{H}(E)]_{\rank E}+\frac{\Delta(E)/(\rank
E)^2}{\mu_{H}(E)-[\mu_{H}(E)]_{\rank E}}\Big)+l(E)\geq\ch_3(E).
\end{equation*}
If $\mu^{\max}_{H}(E)\leq\mu_{H}(E)-\frac{1}{\rank
E}\sqrt{\frac{\Delta(E)}{\rank E+1}}$, by Theorem \ref{main1} and
(\ref{eq7.2}), one has
\begin{equation*}
\frac{(\rank E+2)(\Delta(E))^{\frac{3}{2}}}{(\rank E)^2\sqrt{\rank
E+1}}+\frac{c_1^3(E)-3c_1(E)\Delta(E)}{(\rank E)^2}\geq6\ch_3(E).
\end{equation*}
Thus Theorem \ref{Chern}
follows.
\end{proof}

In particular, when $\rank E=2$, one sees
$[\mu_H(E)]_2=\frac{c_1(E)-1}{2}$. Hence by the formulas
$\ch_2=\frac{1}{2}c_1^2-c_2$ and
$\ch_3=\frac{1}{6}(c_1^3-3c_1c_2+3c_3)$, Theorem \ref{Chern} gives
Corollary \ref{cor}. From $[0]_3=-\frac{1}{3}$,
$[-\frac{1}{3}]_3=-\frac{1}{2}$ and $[-\frac{2}{3}]_3=-1$, we can
also bound $c_3$ for a rank 3 stable sheaf on $\mathbb{P}^3$
(compare it with the bounds got by Ein, Hartshorne and Vogelaar
\cite[Theorem 4.2 and 4.3]{EHV}).

\begin{remark}
In the proof of Theorem \ref{Chern}, if one replaces the plane $H$
by a higher degree divisor $H_d=dH$, then $\beta$ should be changed
into $d\beta$. Thus inequality \ref{eq7.2} becomes
\begin{equation}
\frac{\Delta(E)}{\rank E}\Big(\mu_{H}(E)-\beta d
+\frac{\Delta(E)/(\rank E)^2}{\mu_{H}(E)-\beta d
}\Big)+6l(E)\geq6\ch_3(E).
\end{equation}
On the other hand, one sees that
$$\beta_0=\mu_{H_d}(E)-\frac{\overline{\Delta}_{H_d}(E)/(H_d^3\rank
E)^2}{\mu_{H_d}(E)-\mu^{\max}_{H_d}(E)}=\frac{1}{d}\Big(\mu_{H}(E)-\frac{\Delta(E)/(\rank
E)^2}{\mu_{H}(E)-\mu^{\max}_{H}(E)}\Big),$$ and
$$\beta_1=\mu_{H_d}(E)-\frac{\sqrt{(\rank
E+1)\overline{\Delta}_{H_d}(E)}}{H_d^3\rank
E}=\frac{1}{d}\Big(\mu_{H}(E)-\frac{\sqrt{(\rank
E+1)\Delta(E)}}{\rank E}\Big).$$ Thus we obtain the same
inequalities as those in Theorem \ref{Chern}. That means Theorem
\ref{Chern} does not depend on the degree of $H_d$.
\end{remark}

\begin{proof}[Proof of Corollary \ref{cor1}]
By Lemma \ref{lem}, one deduces
$$\mu_H(E)-\mu^{\max}_{H}(E)\geq\mu_H(E)-[\mu_H(E)]_{\rank
E}\geq\frac{1}{(\rank E)^2}.$$

If $\mu^{\max}_{H}(E)>\mu_{H}(E)-\frac{1}{\rank
E}\sqrt{\frac{\Delta(E)}{\rank E+1}}$, then

$$\mu_H(E)-[\mu_H(E)]_{\rank
E}+\frac{\Delta(E)/(\rank E)^2}{\mu_H(E)-[\mu_H(E)]_{\rank E}}\leq
\frac{1}{(\rank E)^2}+\Delta(E).$$ Hence Theorem \ref{Chern} implies
\begin{equation*} \frac{\Delta(E)}{6\rank
E}\Big(\frac{1}{(\rank E)^2}+\Delta(E)\Big)+l(E)\geq\ch_3(E).
\end{equation*}

Since $\rank E\geq3$, one sees
$$\frac{\Delta(E)}{6\rank E}\Big(\frac{1}{(\rank
E)^2}+\Delta(E)\Big)\geq\frac{(\rank
E+2)(\Delta(E))^{\frac{3}{2}}}{6(\rank E)^2\sqrt{\rank E+1}}.$$ This
completes the proof.
\end{proof}

\bibliographystyle{amsplain}

\end{document}